\documentclass{amsart}
\usepackage{graphicx} 
\usepackage{amsmath}  
\usepackage{amsthm}  
\usepackage{amsfonts}
\usepackage{enumerate}
\usepackage{placeins}   

\usepackage{xurl}   
\usepackage[colorlinks=true,linkcolor=blue,citecolor=blue,urlcolor=blue]{hyperref}   

\usepackage{tcolorbox}

\def\vector#1#2#3{\left(
\begin{array}{l} 
#1\\ 
#2\\ 
#3 
\end{array}\right)}

\def\matrix#1#2#3#4#5#6#7#8#9{{
 \left( \begin{array}{ccc}
#1 & #2 & #3 \\
#4 & #5 & #6 \\
#7 & #8 & #9 
\end{array} \right)}}

\theoremstyle{plain}
\newtheorem{thm}{Theorem}
\newtheorem{lem}[thm]{Lemma}
\newtheorem{prop}[thm]{Proposition}
\newtheorem{cor}[thm]{Corollary}

\theoremstyle{definition}

\newtheorem{exmp}{Example}[section]
\theoremstyle{remark} 
\newtheorem*{remark}{Remark}
\newcommand{\Q}{\mathbb{Q}} 
\newtheorem*{mainrestatement}{{\bf Theorem~\ref{thm:main}}}

\def\al{\alpha}
\def\be{\beta}
\def\ga{\gamma}
\def\Ga{\Gamma}
\def\qq{{\mathbb Q}}
\def\Z{{\mathbb Z}}
\def\ha{\frac{1}{2}}
\def\onehalf{\tfrac{1}{2}}
\def\cd{\cdot}
\def\Erdos{Erd\H{o}s}

\title{Solution of Erd\H{o}s Problem 633}
\author[Beeson]{Michael Beeson}
\address{San Jos\'e State University (emeritus), and UCSC (research associate)}

\author[Laczkovich]{Mikl\'os Laczkovich}
\address{E\"otv\"os Lor\'and University}

\author[Zhang]{Yan X Zhang}
\address{San Jos\'e State University}
\date{\today}
\subjclass[2020]{51M20, 51M04, 11D25}

\begin{document}

\begin{abstract}
It is well-known that any triangle can be \emph{tiled} (cut) into any square number of congruent triangles similar to the original triangle.
We classify all triangles that can be tiled into a non-square number of congruent triangles, settling \Erdos\ Problem 633. Our work shows that, 
aside from isosceles triangles,
this set of triangles is countable (up to similarity).
We also show that (with a few exceptions) if a triangle $T$ has any tiling into congruent triangles not similar to $T$, then every such tiling has a non-square number of tiles.
\end{abstract}
\thanks{The second author was supported by the Hungarian National Foundation
for Scientific Research, Grant No. K146922. The third author wishes to thank
Boris Alexeev for informing him of the existence of this problem.
We used Claude (Anthropic) and Opus (OpenAI) solely for proof-checking and copy editing. 
}
\maketitle

\section{Introduction}

We say that a triangle $T$ \emph{tiles} into a triangle $R$ (the \emph{tile}) if $T$ can be cut into a union of triangles congruent to $R$, overlapping only at their boundaries. Call such a tiling a \emph{square} tiling if the number of tiles is a square and a \emph{reptiling} if $T$ and $R$ are similar.  Every triangle has ``generic'' square reptilings into $N^2$ tiles for any $N$ (see Figure~\ref{figure:quadratic}), so the following (\$25) Erd\H{o}s Problem from \cite[p. 48]{Soifer2009-book} (indexed as Problem 633 in \cite{erdos-633}) can be seen as trying to understand which triangles $T$ can be ``non-generically'' tiled. 
\begin{tcolorbox}
Classify those triangles which can only be tiled by square tilings.
\end{tcolorbox}

Our main result, which settles the problem, is the following. 

\begin{thm} \label{thm:main}
A triangle $T$ admits a non-square tiling if and only if it satisfies one of
the following conditions, where $(A,B,C)$ are the angles of $T$ in some order: 
\begin{enumerate}[{\rm (1)}]
    \item $A=B$, i.e. $T$ is an isosceles triangle (including equilateral);
    \item $C =  \pi / 2$ and the legs of the right triangle $T$ are in integer
      ratio $M/K$, where $M^2 + K^2$ is not a square;
    \item $(A,B,C)  = (\pi/6, \pi/2, \pi/3)$;
    \item $C = \pi/3$, with $\sqrt{3}\tan(A/2)$ rational;
    \item $B = 2A$,  with $\sqrt{3}\tan(A/2)$ rational;
    \item $B = 2A$, with $\sin(A/2)$ rational;
    \item  $C = A/2 + B$, with $2 \sin(A/4)$
    rational, equal to $M/K$,\\ where $2K^2-M^2$ is not a square and $M,K \in \Z$; 
    \item $C = 2A+B/2$, with $\sqrt{3}\tan(A/2)$ rational.
\end{enumerate}
\end{thm}

\begin{cor}
Apart from the isosceles triangles, the number of similarity classes of triangles having
a non-square tiling is countable. 
\end{cor}

We also prove that for most triangles every square tiling must be a reptiling.

\begin{thm}\label{thm:non-reptiling}
Suppose we have a tiling of triangle $T$ that is not a reptiling. Then the
number of tiles cannot be a square unless one of the following is true:
\begin{enumerate}[{\rm (1)}]
\item $T$ is isosceles; 
\item $T$ has angles $(A,B,C)$ with $C = A/2 + B$ and $2 \sin(A/4)$ rational,
equal to $M/K$, where
$2K^2-M^2$ is a square.
\end{enumerate}
\end{thm}

After giving some preliminaries in Section~\ref{sec:preliminaries}, we prove
the  left-to-right direction of Theorem~\ref{thm:main} in
Section~\ref{sec:forward}, showing that one of the given conditions must hold
if we have a non-square tiling. Then in Section~\ref{sec:existence} we prove the
right-to-left direction.  We apply the following two steps to each case:  
\begin{itemize}
\item  The number $N$ of tiles is given by dividing the area of $T$ by the area
of the tile.  That gives a formula for $N$ in terms of the angles and sides
of $T$, differing per case.
\item Apply number theory to prove that $N$ cannot be a square. For the number of tiles
to be a square, we would have integer solutions to Diophantine equations
of the form $N = Z^2$. These equations reduce to elliptic curves. 
As it turns out, the ranks of all these elliptic curves are zero, so it
is enough to find their torsion points, and to check that those points do not
correspond to actual tilings.
\end{itemize}

We prove Theorem \ref{thm:non-reptiling} in Section \ref{sec:uniqueness}.
We end with examples of our tilings in Section~\ref{sec:examples}.

\section{Preliminaries}
\label{sec:preliminaries}

We say that a triangle has \emph{commensurable angles} if all of its angles are rational multiples of $\pi$ and \emph{incommensurable angles} otherwise. The case in which the tiled triangle has commensurable 
angles has been dealt with in \cite{snover1991,laczkovich1995tilings,laczkovich2012tilings}.
The result, however, is not stated in those sources in the exact form we need.
We therefore state it, and extract its proof from the cited
papers.  We start with a lemma:

\begin{lem} \label{lem:RandT} Let a triangle $T$ be tiled by triangle $R$,
where $T$ is not equilateral. Then $R$ has  commensurable angles if and only
if $T$ has commensurable angles.
\end{lem}

\begin{proof} See the first paragraph of the proof of Theorem~5.3 in \cite{laczkovich1995tilings}.
\end{proof}

\begin{thm}[\cite{laczkovich1995tilings,laczkovich2012tilings}]
\label{thm:l1995-5.3} Suppose $T$ has commensurable angles and tiles into $R$.
If $T$ is not an isosceles triangle, then the tiling is a
reptiling.
\end{thm}

\begin{remark} Of course an equilateral triangle is isosceles, but often in this 
subject, the equilateral case is treated separately.  We emphasize that this theorem
excludes equilateral triangles, too.
\end{remark}

\begin{proof} Following \cite{laczkovich1995tilings}, let $c(T)$ be the 
number of distinct non-similar triangles $R$ into which $T$ can be tiled.  Then $c(T) = 1$ 
means that $T$
can only be tiled by $R$ similar to $T$ (that is, a reptiling). 
Theorem~5.3 of \cite{laczkovich1995tilings} says that, under the assumptions that $T$ has
commensurable angles and is not isosceles, or is the isosceles right triangle,
then $c(T) = 1$, and if $T$ is isosceles but neither right nor equilateral, then $c(T) = 2$.
Hence, under the assumptions of our theorem, $c(T) = 1$. 
\end{proof}

The main result related to reptiling is:
\begin{thm}[\cite{snover1991}]
    \label{thm:reptile} Suppose $T$ can be reptiled into $N$ tiles. Then $N$ must be one (or more) of:
    \begin{itemize}
        \item $N = M^2$, which is possible for any triangle $T$ and any natural number $M$;
        \item $N = M^2 + K^2$, in case where $N$ is not a square, $T$ must be a right triangle with legs having ratio $M/K$;
        \item $N = 3M^2$, in which case $T$ must be the $(\pi/2, \pi/3, \pi/6)$ right triangle.
    \end{itemize}
\end{thm}

Combining these two theorems, we obtain:
\begin{cor}
    \label{cor:commensurable} If $T$ has commensurable angles, then $T$ admits a non-square tiling if and only if one of the following is true:
    \begin{enumerate}[{\rm (1)}]
        \item $T$ is isosceles (including equilateral);
        \item $T$ is a right triangle with legs in integral ratio $M/K$ where $M^2+K^2$ is not a square;
        \item $T$ is the $(\pi/2, \pi/3, \pi/6)$ right triangle.
    \end{enumerate}
\end{cor}
\begin{proof}
If $T$ is isosceles, $T$ has at least one non-square tiling into $2$ tiles along its axis of symmetry. If $T$ is not isosceles,  direct applications of Theorems~\ref{thm:l1995-5.3} and \ref{thm:reptile} finish the proof.
\end{proof}

We say that a triangle has \emph{commensurable sides} if its sides have pairwise rational ratios. Equivalently, we can assume the sides are scaled to be integers. The following result allows us to limit our attention to these cases:
\begin{thm}\cite[Theorem 1.2]{rationality}
    \label{thm:rationality}
 Let triangle $T$ be tiled by a tile $R$ such that
    \begin{itemize}
        \item $R$ is not similar to $T$;
        \item $R$ is not a right triangle;
        \item $R$ has incommensurable angles.
    \end{itemize}
    Then $R$ must have commensurable sides.
\end{thm}

\begin{prop}
\label{prop:group-1-alg}
A triangle $R$ with angles $(\alpha, \beta, \gamma)$ where $3\alpha + 2\beta =
\pi$ has commensurable sides if and only if
$$\sin (\alpha /2) \in \mathbb{Q}.$$
\end{prop}

\begin{proof}
Let the sides of $R$ opposite $(\alpha,\beta,\gamma)$ be $(a,b,c)$ respectively. By the law of sines we have $a/c=\sin \alpha /\sin \gamma$
and $b/c=\sin \beta /\sin \gamma$. Thus we have to show that $\sin \alpha /\sin \gamma$ and $\sin \beta /\sin \gamma$ are both rational if and only if
$\sin (\alpha /2)$ is.

We have $\beta =(\pi -3\al )/2$ and $\gamma =(\pi +\alpha )/2$, and thus
$\sin \beta =\cos (3\alpha /2)=4\cos ^3 (\alpha /2) -3\cos (\alpha /2)$ and
$\sin \gamma =\cos (\alpha /2)$. Therefore, we have
$\sin \alpha /\sin \gamma =2\sin (\alpha /2)$ and $\sin \be /\sin \gamma =
4\cos ^2 (\alpha /2) -3=1-4\sin ^2 (\alpha /2)$, from which the statement follows.
\end{proof}

\begin{prop}
\label{prop:group-2-alg}
A triangle $R$ with angles $(\alpha, \beta, \gamma)$ where $\alpha + \beta$
equals $\pi/3$ or $2\pi /3$ has commensurable sides if and only if 
\begin{equation}\label{eq:205} 
  \sqrt{3} \sin \alpha \in \mathbb{Q}  \text{\ \ and\ }  \cos \alpha \in \mathbb{Q},
\end{equation}   
or equivalently, 
\begin{equation}\label{eq:206} \sqrt{3} \tan \tfrac \alpha 2 \in \mathbb{Q}.
\end{equation}
If \eqref{eq:205}  or \eqref{eq:206} holds, then
\begin{equation}  \label{eq:etrig}
\cos \al =\frac{1-3t^2}{1+3t^2} \text{\ \ and \ } \sin \al =\frac{2\sqrt{3} t}{1+3t^2},
\end{equation}
where $t = \tfrac 1{\sqrt 3} \cdot  \tan(\alpha /2) \in \qq$ and 
\begin{equation} \label{eq:268}
0<t<\tfrac 1 3\mbox{\qquad when $\alpha+\beta = \tfrac \pi 3$}. 
\end{equation}
\end{prop}
\begin{proof}  
Let $(a,b,c)$ be the sides of $R$ opposite $(\alpha,\beta,\gamma)$ respectively.
We have $\beta =\pi /3 -\alpha$ or $2\pi /3 -\alpha$, and thus 
$$ \sin \beta =\frac{\sqrt 3} 2 \cos \alpha \pm \frac 1 2 \sin \alpha.$$
By the law of sines we have 
$$\frac{a}{c} = \frac{\sin \alpha} {\sin \gamma} = \frac{2\sin \alpha} {\sqrt 3} = \frac{2\sqrt{3} \sin \alpha}{3} ,$$
and
$$\frac{b}{c}=\frac{\sin \beta} {\sin \gamma}=\cos \alpha \pm
\frac{\sqrt 3}{3} \sin \alpha .$$
Hence $a/c$ and $b/c$ are both rational if and only if $\sqrt 3 \sin \alpha \in \mathbb{Q}$ and $\cos \alpha \in \mathbb{Q}$.  
That proves that $R$ has commensurable 
sides if and only if \eqref{eq:205} holds. 
\smallskip

To show \eqref{eq:206} implies \eqref{eq:205}: let $t=(\tan \tfrac \alpha 2) /\sqrt 3$. 
By \eqref{eq:206}, $t$ is rational.  We have
\begin{eqnarray*}
\cos \al =\frac{1-\tan ^2 \tfrac \alpha 2}{1+\tan ^2 \tfrac \alpha 2} =\frac{1-3t^2}{1+3t^2}  \\
\  \sin \al =\frac{2\tan \tfrac \alpha 2}{1+\tan ^2 \tfrac \alpha 2} =
\frac{2\sqrt 3 t}{1+3t^2}.
\end{eqnarray*}  
Since $t$ is rational, \eqref{eq:205} holds.  Hence \eqref{eq:206} implies \eqref{eq:205}
and \eqref{eq:etrig}.

To show \eqref{eq:205} implies \eqref{eq:206}:  if \eqref{eq:205} holds, then the identity 
$$\tan (\alpha / 2) = \frac{\sin \alpha }{1 + \cos \alpha }$$
shows that $\sqrt{3}\tan(\alpha/2) \in  \mathbb{Q}$, so $t$ is rational. 

Finally we will prove \eqref{eq:268}.  Assume $\alpha+\beta = \pi/3$.  Then $\alpha < \tfrac \pi 3$,
so $$\tan \tfrac \alpha 2  < \tan \tfrac \pi 6 = \frac 1 {\sqrt 3}.$$
\begin{equation*}
 t = \frac {\tan \tfrac \alpha 2}{\sqrt 3} < \frac 1 3. \qedhere
 \end{equation*}   
\end{proof}

\section{The Forward Direction} \label{sec:forward}

\begin{thm}[Corollary of Theorem 4.1, \cite{laczkovich1995tilings}]
\label{thm:l1995-4.1}
Let a non-isosceles triangle $T$ be tiled with a triangle with  angles
$(\alpha, \beta, \gamma)$. Then one of the following must hold:
\begin{enumerate}[{\rm (1)}]
    \item the tiling is a reptiling;
    \item (Group 1): $3\alpha + 2\beta = \pi$ and $T$ has angles $(\alpha, 2\alpha, 2\beta)$ or $(2\alpha, \beta, \alpha + \beta)$; or
    \item (Group 2): $3\alpha + 3\beta = \pi$ (equivalently, $\gamma = 2\pi/3$) and $T$ has angles $$(\alpha, 2\alpha, 3\beta), (\alpha, 2\beta, 2\alpha + \beta), (\alpha, \alpha+\beta, \alpha+2\beta), \text{or }  (2\alpha, 2\beta, \alpha+\beta).$$
\end{enumerate}
\end{thm}

\begin{proof} By Theorem \ref{thm:l1995-5.3}, we may assume that the angles
of $T$ are not commensurable. Therefore, by Lemma \ref{lem:RandT}, $\al$,
$\be$, $\ga$ are not commensurable either, and we  can apply
\cite{laczkovich1995tilings}, Theorem~4.1. That theorem has more cases in the
conclusion than shown here, but does not have our hypothesis that $T$ is not
isosceles. By adding in the requirement that $T$ is not isosceles, we reduce
the number of cases of the original result. \end{proof}

We say that a tiling is \emph{Group 1} or \emph{Group 2} if it is of the corresponding type in the statement of Theorem~\ref{thm:l1995-4.1}.
\begin{lem} \label{lem:T-pov-cases}
If a non-isosceles triangle $T$ has a  non-reptile tiling,
then one of the following four cases must hold for some permutation of the
angles $(A,B,C)$ of $T$:
\begin{itemize}
    \item $C = \pi/3$;
    \item $B = 2A$;
    \item $C = A/2 + B$; 
    \item $C = 2A + B/2$.
\end{itemize}
\end{lem} 

\begin{proof} 
By Theorem~\ref{thm:l1995-4.1}, we must have a reptiling, a Group 1 tiling,
or a Group 2 tiling. We excluded the reptiling case in our assumptions.

In Group 2, $\alpha + \beta = \pi/3$, so in the $(\alpha, \alpha+\beta,
\alpha+2\beta)$ or $(2\alpha, 2\beta, \alpha+\beta)$ cases, we can conclude that
one of the angles, say $C$, equals $\pi/3$. The Group 1 $(\alpha, 2\alpha, 2\beta)$ and Group 2
$(\alpha, 2\alpha, 3\beta)$ both have one angle twice another; in other words,
$B = 2A$. There is one case left in both groups.

In the final Group 1 case, permuting $(A, B, C) = (2\alpha, \beta, \alpha +
\beta)$ gives $C = A/2+B$. In the final Group 2 case, permuting $(A, B, C) =
(\alpha, 2\beta, 2\alpha +\beta)$ gives $C = 2A + B/2$. 
\end{proof}

Thus, our task reduces to carefully classifying when a non-square tiling is
possible in these four cases. We  will first show that the conditions as they
are outlined in Theorem~\ref{thm:main} are necessary. Then, in Section
\ref{sec:existence},  we will show that they are sufficient by
exhibiting actual tilings.

\begin{prop} \label{prop:TtoR}
Let $T$ be a non-isosceles triangle with angles $(A,B,C)$
falling under one of the cases in the first column below. Suppose $T$ has a
non-reptile tiling into tile $R$ with angles 
$(\alpha,\beta,\gamma)$. Then the angles of $T$ are expressed in terms of
$(\alpha,\beta,\gamma)$, and satisfy rationality conditions, 
according to the other columns of the table.

\begin{center}
\begin{tabular}{c c c c}
Angles of $T$ & Rationality & $T$ in terms of $(\alpha,\beta,\gamma)$ & Relations \\
\hline
$C = \pi/3$ & $\sqrt{3}\tan(A/4) \in \Q$  &   $(2\alpha, 2\beta, \alpha + \beta)$   &  $\gamma = 2\pi/3$\\
$C = \pi/3$ & $\sqrt{3}\tan(A/2)\in \Q$  &    $(\alpha, \alpha+2\beta, \alpha+\beta)$ &$\gamma = 2\pi/3$ \\
$B = 2A$ & $\sqrt{3}\tan(A/2)\in \Q$  &$(\alpha, 2\alpha, 3\beta)$ &  $\gamma = 2\pi/3$ \\
$B = 2A$ & $\sin(A/2) \in \Q$          & $(\alpha, 2\alpha, 2\beta)$ &  $3\alpha + 2\beta = \pi$\\
$C = A/2 + B$ & $\sin(A/4)\in \Q$    & $(2\alpha,\beta,\alpha+\beta)$ & $3\alpha + 2\beta = \pi$\\
$C = 2A + B/2$ & $\sqrt{3}\tan(A/2)\in \Q$  & $(\alpha,2\beta,2\alpha+\beta)$& $\gamma = 2\pi/3$
\end{tabular}
\end{center}
\end{prop}

The proof of this proposition will be
given by breaking it into the next four propositions.

\begin{prop}
\label{prop:60-angle-only-if}
Let a non-isosceles $T$ have angles $(A, B, \pi/3)$. Then $T$ has
a non-reptile tiling into $(\alpha, \beta, \gamma)$ only if one 
or both of the following hold:
\begin{enumerate}[{\rm (i)}]
\item $\sqrt 3 \tan(A/4)\in \Q$ and the angles of $T$ are $(2\alpha, 2\beta,
\alpha + \beta)$  
\item $\sqrt 3 \tan(A/2) \in \Q$ and the angles of $T$ are $(\alpha,
 \alpha+2\beta, \alpha+\beta)$.
\end{enumerate}
\end{prop}

\begin{remark} This
covers the first two lines of the table 
in Proposition~\ref{prop:TtoR}.
\end{remark} 

\begin{proof} 
Suppose $T$ has such a tiling. By Theorem \ref{thm:l1995-5.3}, $T$ has
incommensurable angles, and then, by Lemma~\ref{lem:RandT},
$(\alpha,\beta,\gamma)$ are incommensurable as well. By
Theorem~\ref{thm:l1995-4.1} and the non-reptiling assumption, we have a
Group~1 or Group~2 tiling.
 
{\em Case~1}, we have a Group~1 tiling.  Then
 $R = (\alpha,\beta,\gamma)$ satisfies 
 $3\alpha + 2\beta = \pi$,
 and the angles $(A, B, \pi/3)$ are, 
 in some order, $(\alpha, 2\alpha, 2\beta)$
 or $(2\alpha, \beta,\alpha+\beta)$.  
 Then $\pi/3$ is either $\alpha$, $2\alpha$, $2 \beta$, $\beta$, or $\alpha+\beta$. 
  Each of the first four cases implies that one (and thus both) of
 $\alpha$ and $\beta$ is a rational multiple of $\pi$, which contradicts incommensurable angles. The only remaining possibility is $\alpha + \beta = \pi/3$, but then we have $3\alpha + 3\beta = \pi$, which is impossible if we also have $3\alpha + 2\beta = \pi$.
 Therefore the tiling is 
 not a Group~1 tiling.

{\em Case~2}, we have a Group~2 tiling.
 Then $R = (\alpha,\beta,\gamma)$ satisfies 
 $3\alpha + 3\beta = \pi$,  so $\gamma = 2\pi/3$, and some permutation of the angles of $T$, $(A, B, \pi/3)$, is equal to 
$(\alpha, 2\alpha, 3\beta)$, 
$(\alpha, 2\beta, 2\alpha+\beta)$,
$(\alpha, \alpha+\beta, \alpha+2\beta)$, or
 $(2\alpha, 2\beta, \alpha+\beta)$.
 
As before, one of the angles must  be equal to $\pi/3$. If this angle is one of $\{\alpha, 2\alpha, \beta, 2\beta, 3\beta\}$, then $\alpha$ and $\beta$ will both be rational multiples of $\pi$. Furthermore, we cannot have $\pi/3 = 2\alpha + \beta$ (and by the same logic, $\alpha + 2\beta$), since $\pi/3 = \alpha + \beta$. This rules out the first two cases, leaving only the last two cases with the $\pi/3$ angle assigned as $\alpha+\beta$. Then either $(A,B) = (\alpha,\alpha+2\beta)$, or
$(A,B) = (2\alpha, 2\beta)$, up to permutation.

By Theorem~\ref{thm:rationality}, $R$ has commensurable sides. Then $T$ has commensurable sides as well. 
By Proposition~\ref{prop:group-2-alg}, applied to $T$, we find $\sqrt{3} \cd \tan(A/2)\in \Q$ and also $\sqrt{3} \cd \tan(B/2)\in \Q$.%
\footnote{
It is also possible, instead of using Proposition~\ref{prop:group-2-alg}, to verify directly by trigonometry that if $\alpha+\beta = \pi/3$ then
$$\sqrt 3 \tan(\alpha/2) \in \Q \leftrightarrow \sqrt 3 \tan(\beta/2) \in \Q.$$ }
If $(A,B)=(\alpha,\alpha+2\beta)$ or $(A,B) = (\alpha+2\beta,\alpha)$, then we have (ii).
Suppose $(A,B) = (2\alpha, 2\beta)$. By Proposition~\ref{prop:group-2-alg}, applied to $R$, we find $\sqrt{3} \cd \tan(\alpha/2) \in \Q$, so $\sqrt{3}\tan(A/4) = \sqrt{3}\tan(\alpha/2) \in \Q$, and (i) holds. 
\end{proof}

\begin{prop}\label{prop:A-2A-only-if}
Let a non-isosceles $T$ have angles $(A,2A,\pi -3A)$. Then $T$ has
a non-reptile tiling into $(\alpha,\beta,\gamma)$ only if one of 
the following conditions is satisfied:
\begin{enumerate}[{\rm (i)}]
\item the angles of $T$ are $(\alpha,2\alpha,2\beta)$, and $\sin (A/2) \in \qq$.
\item the angles of $T$ are $(\alpha, 2\alpha, 3\beta)$, and $\sqrt 3  \tan (A/2)\in \qq$. 
\end{enumerate}
\end{prop}

\begin{remark} This 
covers the third and fourth lines of the table 
in Proposition~\ref{prop:TtoR}.
\end{remark}

\begin{proof}
Suppose $T$ has such a tiling.  By Theorem \ref{thm:l1995-5.3} and
Lemma~\ref{lem:RandT}, $(\alpha,\beta,\gamma)$ are incommensurable. By
Theorem~\ref{thm:l1995-4.1} and the non-reptiling assumption, we have a Group~1
or Group~2 tiling.

{\em Case~1}, we have a Group~1 tiling.  Then $\be =(\pi -3\al )/2$ and 
$$\{ A,2A, \pi -3A\} =\{ \al ,2\al, 2\be \} \ \text{or} \ \{2\alpha, \beta,
\alpha + \beta\}.$$
This implies $A=\al$. Indeed, otherwise  $\al /\be \in \{ 1,2, \tfrac 1 2,
\tfrac 1 3 , \tfrac 1 4\}$ would hold, and then $\al ,\be$ would be rational
multiples of $\pi$, which is impossible.  Then by Theorem~\ref{thm:rationality}
and Proposition~\ref{prop:group-1-alg}, $\sin(\alpha/2) \in \Q$. That is, (i)
holds.

{\em Case~2}, we have a Group~2 tiling. Then $\ga =2\pi /3$ and 
$$\{ A,2A, \pi -3A\} =
\{\alpha, 2\alpha, 3\beta\}, \{\alpha, 2\beta, 2\alpha + \beta\}, \{\alpha, \alpha+\beta, \alpha+2\beta\}, \text{or }  \{2\alpha, 2\beta, \alpha+\beta\}.
$$
We will show that $A= \al$, by considering each of the four possible values of $\{ A,2A, \pi -3A\}$.
First, suppose $\{ A,2A, \pi -3A\} = \{\al, 2\al,3\beta\}.$
Then $$(A,2A,\pi-3A) = (\al, 2\al, 3\beta),$$ since every other permutation of the angle either makes
$\alpha/\beta$ rational, or makes $A=2\al$ and $2A = \al$, which is impossible.

Second, suppose $\{ A,2A, \pi -3A\}=\{\alpha, 2\beta, 2\alpha + \beta\}$.  No two of those
last three angles have a rational ratio, so this case is impossible.

Third, suppose $\{ A,2A, \pi -3A\}= \{\alpha, \alpha+\beta, \alpha+2\beta\}$. No two of those
last three angles have a rational ratio, so this case is impossible.

Fourth, suppose $\{ A,2A, \pi -3A\}= \{2\alpha, 2\beta, \alpha+\beta\}$. No two of those
last three angles have a rational ratio, so this case is impossible.

The only possibility not ruled out was $( A,2A, \pi -3A) = (\alpha, 2\alpha, 3 \beta)$.
By Theorem~\ref{thm:rationality} and Proposition~\ref{prop:group-2-alg},
$\sqrt 3 \tan(\alpha/2) \in \Q$. That is, (ii) holds.
\end{proof}

\begin{prop}\label{prop:group-1-last-only-if}
Let a non-isosceles $T$ have angles $(A, B, A/2+B)$. Then $T$ has
a non-reptile tiling into $(\alpha,\beta,\gamma)$ only if the 
angles of $T$ are $(2\alpha,\beta,\alpha+\beta)$ and $\sin(A/4) = \sin(\alpha/2)
\in \mathbb{Q}$.
\end{prop}
\begin{remark} This 
covers the fifth line of the table 
in Proposition~\ref{prop:TtoR}.
\end{remark}

\begin{proof}
Suppose $T$ has such a tiling. By Theorem \ref{thm:l1995-5.3} and
Lemma~\ref{lem:RandT}, $(\alpha,\beta,\gamma)$ are incommensurable. By
Theorem~\ref{thm:l1995-4.1} and the non-reptiling assumption, we have a Group~1
or Group~2 tiling, where $A = p \alpha + q \beta$, $B = s\alpha + t\beta$, and
$p,q,s,t$ are nonnegative integers. Then
\begin{equation}
\label{eq:group-1-last}
    2\pi = 2(A + B + (A/2+B)) = (3p + 4s) \alpha + (3q+4t) \beta.
\end{equation} 
Since $\alpha$ and $\beta$ are linearly independent over $\Q$ (else we
would have commensurable angles), any other expression of $2\pi = x\alpha +
y\beta$ must force $3p+4s = x$ and $3q+4t = y$. We proceed in two cases: 

{\em Case~1}, the tiling is in Group~1. Then $3\alpha + 2\beta = \pi$, so
$2\pi = 6\alpha + 4\beta$ and by \eqref{eq:group-1-last} we must have $3p + 4s = 6$ and $3q+4t = 4$. This forces $(p, q) = (2, 0)$ and $(s,t) = (0, 1)$. In other words, $A = 2\alpha$, $B=\beta$, $C = \alpha+\beta$. By Theorem~\ref{thm:rationality} and Proposition~\ref{prop:group-1-alg}, $\sin(\alpha/2) = \sin(A/4)$ must be rational.

{\em Case~2}, the tiling is in Group~2. Then $3\alpha + 3\beta = \pi$,
so $2\pi = 6\alpha + 6\beta$, and by \eqref{eq:group-1-last} we have $3p+4s = 6$ and $3q+4t=6$. This forces $s = t = 0$, which is impossible since then $B = 0$.
\end{proof}

\begin{prop}\label{prop:group-2-last-only-if}
Let a non-isosceles $T$ have angles $(A,B,2A+B/2)$. Then $T$ 
has a non-reptile tiling into $(\alpha, \beta, \gamma)$ only if 
the angles of $T$ are $(\alpha,2\beta,2\alpha+\beta)$ and 
$\sqrt{3}\tan(A/2) \in \mathbb{Q}$. 
\end{prop}

\begin{remark} This 
covers the sixth (and last) line of the table 
in Proposition~\ref{prop:TtoR}.
\end{remark}

\begin{proof}
Suppose $T$ has such a  tiling. By Theorem \ref{thm:l1995-5.3} and
Lemma~\ref{lem:RandT}, $(\alpha,\beta,\gamma)$ are incommensurable. By
Theorem~\ref{thm:l1995-4.1} and the non-reptiling assumption, we have a Group~1
or Group~2 tiling, where $A = p \alpha + q \beta$, $B = s\alpha + t\beta$,
and $p,q,s,t$ are nonnegative integers. Then
\begin{equation}
\label{eq:group-2-last}
    2\pi = 2(A + B + (2A+B/2)) = (6p + 3s) \alpha + (6q+3t) \beta.
\end{equation} 
Since $\alpha$ and $\beta$ are linearly independent with each other (else we
would have commensurable angles), any other expression of $2\pi = x\alpha +
y\beta$ must force $6p+3s = x$ and $6q+3t = y$. We proceed in two cases: 

{\em Case~1}, the tiling is in Group~1. Then $3\alpha + 2\beta = \pi$, so
$2\pi = 6\alpha + 4\beta$ and by \eqref{eq:group-2-last} we must have $6p + 3s = 6$ and $6q+3t = 4$. The second equation cannot be satisfied, so this case is impossible.

{\em Case~2}, the tiling is in Group~2. Then $3\alpha + 3\beta = \pi$, so
$2\pi = 6\alpha + 6\beta$, and by \eqref{eq:group-2-last} we have $6p+3s = 6$
and $6q+3t=6$. 
Each of these equations has exactly two solutions in nonnegative
integers, namely $(p,s) = (1,0)$ or $(0,2)$, and $(q,t) = (1,0)$ or $(0,2)$.
Taking $(p,s) = (q,t) = (1,0)$ gives $B = 0$, and taking $(p,s) = (q,t) = (0,2)$
gives $A = 0$; both are impossible.  Hence we must either have:
\begin{itemize}
    \item $(p,q) = (1,0), (s,t) = (0, 2)$, giving $(A, B, 2A+B/2) = (\alpha, 2\beta, 2\alpha+\beta)$.
    \item $(p,q) = (0,1), (s,t) = (2,0)$, giving $(A, B, 2A + B/2) = (\beta, 2\alpha, 2\beta + \alpha)$, same as the previous by symmetry.
\end{itemize}
By Theorem~\ref{thm:rationality} and Proposition~\ref{prop:group-2-alg}, we 
have $\sqrt 3 \tan(\alpha/2) \in \Q$ and $\sqrt{3} \cdot \tan (\beta /2) \in \Q$,
and thus $\sqrt{3} \cdot \tan (A/2) \in \Q$.
\end{proof}

\begin{mainrestatement}[Left to right direction]
Suppose triangle $T$ admits a non-square tiling. Then $T$ satisfies at least
one of the following conditions,  where $(A,B,C)$ are the angles of $T$ in some
order: 
 \begin{enumerate}[{\rm (1)}]
    \item $A=B$, i.e. $T$ is an isosceles triangle (including equilateral);
    \item $C =  \pi / 2$ and the legs
    of the right triangle $T$ are in integer ratio $M/K$, \\ where $M^2 + K^2$ is not a square;
    \item $(A,B,C)  = (\pi/6, \pi/2, \pi/3)$;
    \item $C = \pi/3$, with $\sqrt{3}\tan(A/2)$ rational;
    \item $B = 2A$,  with $\sqrt{3}\tan(A/2)$ rational;
    \item $B = 2A$, with $\sin(A/2)$ rational;
    \item  $C = A/2 + B$, with $2 \sin(A/4) = M/K$, where \\
    $2K^2-M^2$ is not a square and $K,M \in \Z$;
    \item 
    $C = 2A+B/2$, with $\sqrt{3}\tan(A/2)$ rational.
    \end{enumerate}
\end{mainrestatement}

\begin{proof}  Let $T$ be tiled by $N$ tiles, where $N$ is not a square.
Suppose first that it is a reptiling. By Theorem~\ref{thm:reptile}, since $N$ is not a square,
one of (2) or (3) is true. Thus we can
assume that we do not have a reptiling.  If $T$ is isosceles, condition (1) holds; so
we may assume $T$ is not isosceles. 
By Lemma~\ref{lem:T-pov-cases}, we
conclude that at least one of the remaining cases must hold. Reading from
the table in Proposition~\ref{prop:TtoR}, we see that each case implies one of
the conditions in the problem statement:
\begin{enumerate}
\renewcommand{\labelenumi}{(\roman{enumi})}
    \item $C = \pi/3$, implying (4);
    \item $B = 2A$,  implying (5) or (6);
    \item $C = A/2 + B$, implying the first part of (7); 
    \item $C = 2A + B/2$, implying (8). 
\end{enumerate}
It remains to check the second part of (7). Assume $C = A/2 + B$.  
  Let $(a,b,c)$ be the sides of the tile and $(\alpha,\beta,\gamma)$ its angles.
Proposition~\ref{prop:TtoR} also tells us that the angles of $T$ are $(2\alpha,\beta, \alpha+\beta)$.
According to Theorem~4 of \cite{beeson-isosceles}, there exist integers $K,M$ such 
that 
$$ 2 \sin \frac \alpha 2 = \frac M K \text{\ and \ } N = 2K^2- M^2.$$
 Since $A = 2 \alpha$, we have $2\sin(A/4) = M/K$. That completes the proof of (7).
\end{proof}

\section{Some number theory}

In order to prove the ``if'' part of Theorem \ref{thm:main} we will show that

\begin{enumerate}
\item tilings with the parameters given in the table of Proposition \ref{prop:TtoR} do exist;
\item there are rational parametrizations for the number of tiles required,
per case of tiling type;
\item those rational functions do not take square values.
\end{enumerate}
The first two parts will be accomplished by carefully combining previous
results. The third part comes down to showing certain Diophantine equations
have no rational solutions, and this is the content of the present section.

First, we cite a useful transformation from \cite{cohen2007number}.
 
\begin{lem}[Corollary 7.2.2, p.~477, \cite{cohen2007number}]
\label{lem:7.2.2}
Each rational solution $(t,s)$ to $$s^2 = t^4 + at^2 +b$$ gives a rational solution $(x,y)$ to $$y^2 = x^3 -2a x^2 +(a^2 -4b)x,$$
via the transformation
\begin{eqnarray*}  
x&=& 2t ^2 -2s +a,  \\
y&=&2t(2t^2 -2s+a).
\end{eqnarray*}
\end{lem}
\begin{proof}
The algebra is routine. If $x \neq 0$, we have the inverse 
\begin{eqnarray*} 
t&=&\frac{y}{2x}, \\
s&=&\left( \frac{y}{2x} \right) ^2 -(x-a)/2.
\end{eqnarray*}
\end{proof}

\begin{prop} \label{prop:nt-squares}
If $a$ and $b$ are positive integers, then $a^2 + b^2 + ab$ and $a(a+b)$ cannot
both be squares. 
\end{prop}
\begin{proof} Let $d= \gcd(a,b)$. Since replacing $a$ by $a/d$ and $b$
by $b/d$ does not affect the conclusion of the Proposition, we may assume
that $a$ and $b$ are coprime.
  
Let $a^2 + b^2 + ab = c^2$. Suppose $a(a+b)$ were a square. Since $a$ and $b$ are coprime, $\gcd(a, a+b) = \gcd(a, b) = 1$. This means $a$ and $(a+b)$ must both be (coprime) squares. Suppose $a = m^2$ and $a+b = n^2$, with $\gcd(m,n) = 1$. Then $b = n^2 - m^2$ and we may substitute $a$ and $b$ into $c^2 = a^2 + b^2 + ab$ to get
\begin{equation}
    \label{eq:classic}
    c^2 = n^4 - m^2n^2 + m^4.
\end{equation}
By \cite[p.~638]{dickson1999history}, the only positive solutions to
\eqref{eq:classic} are of the form $(m,n,c) = (t,t, t^2)$, which would force $b = 0$, contradicting the hypothesis that $b$ is positive. 

While we are done, we provide another proof that demonstrates our common strategy for the remaining cases where the equations are not (at least immediately to us) well-known. First, via the substitution $t = n/m$ and $s = c/m^2$, we know that each rational solution $(m,n,c)$ of \eqref{eq:classic} gives a solution $(t,s)$ to the equation
\begin{equation}
    \label{eq:classic-dehomog}
    s^2 = t^4 - t^2 + 1.
\end{equation}
By Lemma~\ref{lem:7.2.2}, rational solutions of \eqref{eq:classic-dehomog} give rational solutions to the equation
\begin{equation}\label{eell5}
x^3 +2x^2 -3x = y^2.
\end{equation}  
Let $\Ga$ denote the group of rational points of the curve \eqref{eell5}. The
rank of $\Ga$ is zero; this can be shown by applying the method described in
\cite[pp.~91-94]{SilvermanTate1992}.  By the Nagell-Lutz Theorem
\cite[p.~56]{SilvermanTate1992}, we only need to check a finite number of
points. We can check that there are $8$ elements of $\Ga$ of finite order (and
thus the whole group): $$(0,0), (1,0), (-3,0), (-1,2), (-1,-2), (3,6), (3,-6)$$
and the point of infinity. Our claims can be checked as properties of the
elliptic curve with label 24.a4 at LMFDB \cite{lmfdb:24.a4}, with the minimal
Weierstrass form $$y^2=x^3-x^2-4x+4.$$ 

By reversing our transformation, we find that \eqref{eq:classic-dehomog} only
has the following solutions: $(t,s)= (0,\pm 1), (\pm 1, \pm 1)$. The former
would imply 
$n= 0$ in \eqref{eq:classic}, which we disallow since $n^2 -m^2 = b > 0$.
The latter would imply $m=n$, which we disallow since $n^2 -m^2 = b > 0$.

\end{proof}

\begin{prop} \label{prop:nonsquare3}
If $t$ is rational, then 
$$(t^2-2)(t^2-3)$$ is not a square of a rational number.
\end{prop}

\begin{proof} 
By Lemma~\ref{lem:7.2.2}, rational solutions to $(t^2-2)(t^2-3) = s^2$ can be mapped to rational solutions on 
\begin{equation}\label{eell3}
x^3 +10x^2+x=y^2.
\end{equation}  
Let $\Ga$ denote the group of rational points of the curve \eqref{eell3}. As
in the proof of Proposition~\ref{prop:nt-squares}, we can check that the rank
of $\Ga$ is zero by applying the method described in
\cite[pp.~91-94]{SilvermanTate1992}, and then use the Nagell-Lutz theorem and
check that the only elements of $\Ga$ are $(0,0)$ and the point at infinity.
Our claims can be checked as properties of the elliptic curve with label 96.b1
at LMFDB \cite{lmfdb:96.b1}, with the minimal Weierstrass form
$x^3+x^2-32x+60=y^2$. 

Since our only rational point $(x,y)$ of \eqref{eell3} must be $(0,0)$,
considering the transformation in Lemma~\ref{lem:7.2.2}, we must have 
$$ 0 = x = 2t^2 - 2s - 5,$$
so $s = (2t^2 - 5)/2$. Then $$s^2 = (4t^4- 20t^2 + 25)/4 = t^4 - 5t^2 + 25/4.$$
Since we knew $s^2 = t^4 - 5t^2 + 6$, this implies $6= 25/4$, a contradiction.
\end{proof}

\begin{prop}
\label{prop:nonsquare-A-2A-2}
Let $t$ be rational with $t\ne 1, -1/3$.  Then 
$$N = \frac{2}{3}\cd \frac{3t^2 -1}{(3t+1) (t-1)}$$
is not the square of a rational number.
\end{prop}

\begin{proof}
Putting $x=(9t+3)/(t-1)$, we have $t=(x+3)/(x-9)$, $x\ne 0,9$, and 
$$N =\frac{x^2 +18x-27}{36x} .$$
Therefore, if $N$ is the square of a rational number, then the elliptic curve
\begin{equation}\label{eell2}
y^2 = (x^2 +18x-27)x
\end{equation}  
has a solution with $x,y\in \qq$ with $x\ne 0$. Let $\Ga$ denote the group of
rational points of the curve \eqref{eell2}. As in the proof of
Proposition~\ref{prop:nt-squares}, we can check that the rank of $\Ga$ is zero
by applying the method described in \cite[pp.~91-94]{SilvermanTate1992}, and
use  the Nagell-Lutz theorem to check that the only elements of $\Ga$ are
$(0,0)$ and the point at infinity. Because  $x\ne 0$, $N$ is not the square of
a rational number. 

Finally, our claims can be checked as properties of the elliptic curve with label 144.a1 at LMFDB \cite{lmfdb:144.a1}.
\end{proof}

\begin{prop}
\label{prop:nonsquare4}
Let $t$ be rational with $t\ne 0, 1, -1/3$. Then 
$$N = \frac{3t^2-6t-1}{(t-1)(3t+1)}$$ is not a square of a rational number. 
\end{prop}

\begin{proof}
Putting $x=(3t+1)/(1-t)$, we have $t=(x-1)/(x+3)$, $x\ne 0,1,-3$,  and
$$N=\frac{x^2 +6x -3}{4x} ,$$
and so $N$ is the square of a rational number if and only if the elliptic curve
\begin{equation}\label{eell}
(x^2 +6x -3)\cd x=y^2 
\end{equation}  
has a solution with $x,y\in \qq$, $x\ne 0,1,-3$.  Let $\Ga$ denote the
group of rational points of the curve \eqref{eell}. As in the proof of
Proposition~\ref{prop:nt-squares}, we can check that the rank of $\Ga$ is zero
by applying the method described in \cite[pp.~91-94]{SilvermanTate1992}, 
and use the Nagell-Lutz theorem to check that the only elements of $\Ga$ are
\begin{equation*}
(-3,6), \ (-3,-6),\ (1,2),\ (1,-2),\ (0,0),
\end{equation*} 
and the point at infinity; these elements form a cyclic group of order $6$
with generator $(-3,6)$. We observe that $x\in \{ 0,1,-3\}$ for every rational
point $(x,y)$ of $\Ga$ and thus, as we have $x\ne 0,1,-3$, $N$
is not the square of a rational number.

Finally, our claims can be checked as properties of the elliptic curve with label 36.a2 at LMFDB \cite{lmfdb:36.a2}.
\end{proof}

\section{Existence} \label{sec:existence}

In this section, we prove the right-to-left direction of Theorem~\ref{thm:main}.
We start with two classical lemmas about rational values of trig functions at rational multiples of $\pi$.

\begin{lem} \label{lem:niven}
Let $\theta$ be a rational multiple of $\pi$ with $\cos \theta \in \qq$.
Then $$\cos \theta \in \{0,\pm \tfrac 12, \pm 1\}.$$
\end{lem}

\begin{proof} See \cite[Corollary 3.12, p. 41]{niven}.
\end{proof}

\begin{lem} \label{lem:niventan}
Let $\theta$ be a rational multiple of $\pi$ with $\tan ^2 \theta \in \qq$.
Then $$\tan ^2 \theta \in \{0,\tfrac 13 ,1,3\}$$
\end{lem}

\begin{proof}
Since $\cos 2\theta = (1-\tan ^2 \theta)/
(1+\tan ^2 \theta)$, $\cos 2 \theta$ is rational. Since $2\theta$ is a rational multiple of $\pi$,
 $\cos 2\theta \in \{0,\pm \tfrac 12 ,\pm 1\}$ by Lemma~\ref{lem:niven}.
Solving $$\tan ^2 \theta = (1-\cos 2\theta)/(1+\cos 2\theta)$$ for these five
values gives $\tan ^2 \theta \in \{0,\tfrac 13, 1, 3\}$, the value
$\cos 2\theta = -1$ being excluded since then $\tan \theta$ is undefined.
\end{proof}

\begin{mainrestatement}[Right to left direction]
For each of the following conditions, suppose that a triangle $T$ with angles $(A,B,C)$
in some order satisfies the condition. Then $T$ admits a non-square tiling.
\begin{enumerate}[{\rm (1)}]
\item $A=B$, i.e. $T$ is an isosceles triangle (including equilateral);
\item $C =  \pi / 2$ and the legs of the right triangle $T$ are in integer
ratio $M/K$, where $M^2 + K^2$ is not a square;
\item $(A,B,C)  = (\pi/6, \pi/2, \pi/3)$;
\item $C = \pi/3$, with $\sqrt{3}\tan(A/2)$ rational;
\item $B = 2A$,  with $\sqrt{3}\tan(A/2)$ rational;
\item $B = 2A$, with $\sin(A/2)$ rational;
\item  $C = A/2 + B$, with $2 \sin(A/4)$ rational, equal to $M/K$, where
$2K^2-M^2$ is not a square. 
\item $C = 2A+B/2$, with $\sqrt{3}\tan(A/2)$ rational.
\end{enumerate}
\end{mainrestatement}

\begin{proof}
For case (1), simply split $T$ symmetrically
into two triangles. 

For case (2), let $N= M^2 + K^2$. Then by hypothesis $N$ is not a square.
An appropriately-chosen right triangle can be $N$-tiled by right triangles
similar to $T$.  See Figure~\ref{figure:biquadratic}.
These {\em biquadratic} tilings were first described by S.W. Golomb in \cite{golomb1964replicating},
and mentioned again in \cite{snover1991}.

For case (3), see Figure~\ref{figure:tiling3} for a tiling with $N=3$, which is not a square.

In each of the cases (4)--(8) we may assume that the angles of $T$ are
incommensurable.  Indeed, suppose they are commensurable.  In cases (4), (5),
and (8) we have $\sqrt 3 \tan (A/2) \in \qq$, so $\tan ^2 (A/2) \in \qq$, and
by Lemma~\ref{lem:niventan}, $\tan ^2 (A/2) \in \{0,\tfrac 13,1,3\}$; the value
$1$ is excluded since $\sqrt 3 \notin \qq$, so $A \in \{\pi/3, 2\pi/3\}$.  In
case (6) we have $\sin (A/2) \in \qq$, so $\cos A = 1-2\sin ^2 (A/2) \in \qq$
and Lemma~\ref{lem:niven} gives $A \in \{\pi/3, \pi/2, 2\pi/3\}$; but
$\sin (A/2) \in \qq$ then forces $A = \pi/3$.  In case (7), $2\sin (A/4) \in \qq$
similarly forces $A = 2\pi/3$.  Now $B = 2A$ in cases (5) and (6) leaves
$C \le 0$; in case (7), $C = A/2+B$ with $A = 2\pi/3$ leaves $B = 0$; in case
(8), $C = 2A+B/2$ leaves $B \le 0$; and in case (4), $C = \pi/3$ with
$A \in \{\pi/3,2\pi/3\}$ leaves only $A = B = C = \pi/3$, which is case (1).
Thus in cases (5), (6), (7), and (8) no such $T$ exists, and in case (4) the
only one is equilateral, already covered.  Hence, as claimed, we may
assume that the angles of $T$ in cases (4)--(8) are incommensurable.

The incommensurable instances of cases (4)--(8) are established,
respectively, by Propositions~\ref{prop:60-angle-existence}--\ref{prop:group-2-last-existence} below.
In case (4),
the condition $\sqrt 3 \tan (A/2) \in \qq$ holds for both of the angles other
than $C$, so we may choose the permutation with $A<B$.
For each of
the cases, we must show that at least one non-square tiling exists. Cases (4),
(5), and (8) have explicit constructions by the third author in
\cite{zhang2026}. Cases (6) and (7) have explicit constructions by the first
author in \cite{beeson-triangletiling3}. However, the existence of at least one
tiling for each triangle falling under cases  (4)-- (8) was already proven by
the second author in \cite{laczkovich1995tilings}. The main work we do is
showing that the number of tiles in  these constructions is not a square.
\end{proof}

\begin{lem} \label{lem:prop24helper} 
If $\alpha+\beta = \pi/3$, then
\begin{equation*}
\sin(\alpha+2\beta) = \sin \alpha + \sin \beta.
\end{equation*}
\end{lem}

\begin{proof}
\begin{eqnarray*}
\alpha &=& \pi/3 -\beta  \\
\alpha + 2 \beta &=& \pi/3 + \beta  \\
\sin(\alpha+2\beta) &=& \sin(\pi/3+\beta)  \\
&=& \sin(\pi/3) \cos \beta + \cos(\pi/3) \sin \beta  \\
&=&  \frac {\sqrt 3} 2 \cos \beta + \frac 1 2 \sin \beta\\
&=& \sin(\pi/3 - \beta) + \sin(\beta)  \\
&=& \sin(\alpha) + \sin \beta \qedhere
\end{eqnarray*}
\end{proof} 
\begin{remark} The preceding lemma can also be proved
more geometrically, using the law of sines, introducing $a$ and $b$
as illustrated in the diagram.

\begin{minipage}[t]{0.55\textwidth}
\begin{eqnarray*}
 \frac{\sin(\alpha+2\beta)}{a+b} &=& \frac {\sin \alpha} a  \mbox{\qquad (law of sines)} \\
 \sin(\alpha+2\beta) &=& a \frac{\sin \alpha} a + b \frac{\sin \alpha} a \\
 &=& \sin \alpha + \sin \beta \qedhere
 \end{eqnarray*}
 
\end{minipage}
\begin{minipage}[t]{0.4\textwidth}
{\centering
\vskip0.1cm
\includegraphics[width=0.75\textwidth]{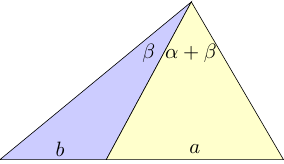}
\par
}
\end{minipage}
\end{remark}
\begin{prop} \label{prop:60-angle-existence}
Let $T$ have incommensurable angles $(A, B, \pi/3)$ with $A<B$, such that
$$\sqrt{3}\tan\left(A/2\right)\in \mathbb{Q}.$$
Then
\begin{enumerate}[{\rm (1)}]
\item $T$ has a tiling into $(\alpha, \beta, \gamma)$ with $\alpha = A$,
$\beta = \pi/3-A$, and $\gamma = 2\pi/3$.
\item For any such tiling, the number of tiles is not a square.
\end{enumerate}
\end{prop}

\begin{remark} By Proposition~\ref{prop:group-2-alg} (or direct trigonometry) 
$$\sqrt{3}\tan\left(A/2\right)\in \mathbb{Q} \iff \sqrt{3}\tan\left(B/2\right)\in \mathbb{Q}.$$
Therefore the hypothesis $A<B$ only fixes which angle is called ``$A$'' and which is called ``$B$.''
\end{remark}

\begin{proof}
By assumption and
Proposition~\ref{prop:group-2-alg}, $\sqrt 3 \sin A$ and $\cos A$ are rational.
Let $\alpha = A, \beta = \pi/3-A$. Then $T = (\alpha, \alpha + 2\beta,
\alpha + \beta)$ and by Theorem~2.5 of \cite{laczkovich1995tilings},
there is a tiling of $T$ into $R$. That proves (1).

Proof of (2): Consider  any such tiling of $T$. We will show that the number of
tiles is not a square. By Proposition \ref{prop:group-2-alg},
the tile has commensurable sides, so by scaling we may assume the sides of the
tile, $a$, $b$, and $c$ are integers.  Since $\gamma = 2\pi/3$,
we have $c^2 = a^2 + b^2 + ab$, and thus $a^2 + b^2 + ab$ is a square.

By Lemma~\ref{lem:prop24helper} and the law of sines, the sides of $T$ facing
angles $(\alpha, \alpha + 2\beta, \alpha + \beta)$ are in ratio of
$(a, a+b, c)$. Therefore, the lengths of the triangle $T$ must be  
$$(am, (a+b)m, cm)$$
for some number $m$. Since the sides of $T$ must be integral, $m$ must be
rational. Then the number of tiles $N$ is the
area of $T$ divided by the area of the tile:
$$N = \frac{\frac{1}{2} am \cd (a + b)m \cd \sin (\pi/3)}{\frac{1}{2}a
\cd b \cd \sin (2\pi/3)} =\frac{a+b}{b}m^2.$$
For this to be a square, $b(a+b)m^2$ and hence $b(a+b)$ also must be squares. 
By Proposition~\ref{prop:nt-squares}, this is not possible.
\end{proof}

\begin{prop}
\label{prop:A-2A-existence1}
Let $T$ have incommensurable angles $(A, 2A, \pi - 3A)$,
and suppose
$\sin(A/2) \in \qq$.  Then,
\begin{enumerate}[{\rm (1)}]
    \item $T$ has a tiling into a tile with angles $(\alpha, \beta, \gamma)$, where $\alpha = A$ and $3\alpha + 2\beta = \pi$.
    \item For any such tiling, the number of tiles cannot be a square.
\end{enumerate}
\end{prop}

\begin{proof}
Suppose $\sin (A/2)\in \qq$.
Put $\al =A$, $\be =(\pi -3\al )/2$, $\ga =(\pi +\al )/2$. Since
$\sin (\al /2)\in \qq$, it follows from \cite[Theorem 2.4]{laczkovich1995tilings}  that $T$ can be
tiled by congruent triangles of angles $(\al ,\be ,\ga )$. Note that
$\pi -3A=2\be$. 

Theorem 13 of 
\cite{beeson-triangletiling3}  says that the number $N$ of tiles is given by 
\begin{eqnarray}
N = M^2 \frac{(2-s^2)(3-s^2)}{(1-s)^2(2+s)^2} 
\mbox{\quad where $s = 2\sin(\alpha/2) = a/c$}
\label{eq:817}
\end{eqnarray}
for some integer $M$.  If $N$ is a square,
then $(2-s^2)(3-s^2)$ is a square.
But that contradicts Proposition~\ref{prop:nonsquare3}.
\end{proof}

\begin{remark}
We will numerically check the formula in the proof above for the 77-tiling 
shown in Figure~\ref{figure:tiling77}. There $\sin (\al /2)=1/4$, $N=77$, $M=5$,
$s=\frac 1 2 $, $(a,b,c) = (2,3,4)$. Both sides of \eqref{eq:817} evaluate to
$77$, so that checks. $\checkmark$
\end{remark}

\begin{prop}
\label{prop:A-2A-existence2}
Let $T$ have incommensurable angles $(A, 2A, \pi - 3A)$.
Suppose  $\sqrt{3} \tan(A/2) \in \qq$. Then, 
\begin{enumerate}[{\rm (1)}]
    \item $T$ has a tiling into $(\alpha, \beta, \gamma)$ with $\alpha = A$, $\beta = \pi/3-A$, and $\gamma = 2\pi/3$.
    \item For any such tiling, the number of tiles is not a square.
\end{enumerate}
 
\end{prop}

\begin{proof}
Put $\al =A$,
$\be =\pi /3 -\al$ and $\ga =2\pi /3$. Suppose $\sqrt 3 \tan (A/2)\in \qq$. By
the assumption and Proposition~\ref{prop:group-2-alg}, we have $\cos \al ,
\sqrt 3 \sin \al \in \qq$. Then $T$ can be tiled by congruent triangles of
angles $(\al ,\be , \ga )$ \cite[Theorem 2.5]{laczkovich1995tilings}. 

We now show the number of tiles cannot be a square. We may assume that the
sides of $R$ and of $T$ are integers. Then the area of $R$ is $\onehalf
\sin \al \cd \sin \ga / \sin \be$ times a square. The area of $T$ is
$\onehalf \sin \al \cd \sin 2\al / \sin 3\al$ times a square. Note that
$\sin 3\al =\sin \al \cd (2\cos \al -1)(2\cos \al +1)$. Thus the number of
tiles is
\begin{align*}
N =& \frac{\sin 2\al \cd \sin \be}{\sin 3\al  \cd \sin \ga}=
\frac{2\sin \al \cd \cos \al \cd \sin \be}{\sin \al \cd (2\cos \al -1 ) (2\cos \al +1 ) \cd \sin \ga}=\\
& \frac{2\cos \al \cd \sin ((\pi /3)-\al )}{(2\cos \al -1)(2\cos \al +1)}\cd \frac{2}{\sqrt 3}=
\frac{2\cos \al \cd (\cos \al -(\sin \al )/\sqrt 3 )}{(2\cos \al -1)(2\cos \al +1)}
\end{align*}
times a square of a rational number. Set $t = \frac 1 {\sqrt 3} \tan(\alpha/2)$.
By Proposition~\ref{prop:group-2-alg}, we have $0 < t < \frac 1 3$, and applying (2) (from
Proposition~\ref{prop:group-2-alg}), computation yields
$$N=\frac{2}{3}\cd \frac{3t^2 -1}{(3t+1) (t-1)}.$$ 
By Proposition~\ref{prop:nonsquare-A-2A-2}, this cannot be a square; the
hypothesis $t\ne 1,-1/3$ of that proposition  is fulfilled since
$0 < t < 1/3$.
\end{proof}

\begin{prop}
\label{prop:group-1-last-existence}
Let $T$ have incommensurable angles $(A, B, C)$ with $C = A/2 + B$ and
$2\sin(A/4) \in \Q$, equal to $M/K$. Then  
\begin{enumerate}[{\rm (1)}]
\item $T$ has a tiling into $(\alpha, \beta, \gamma)$ with $\alpha = A/2$
and $\beta = B$.
\item For any such tiling, the number of tiles is a square if and only if $2K^2 - M^2$ is a square.
\end{enumerate}
\end{prop}

\begin{remark} $N = 2K^2 - M^2$ will be the number of tiles, and the tile
$(\alpha,\beta,\gamma)$ will have $\alpha = A/2$ and $3\alpha + 2\beta = \pi$.
The condition that $N$ is not a square cannot be dropped.
  
For example, $(M,K,N) = (5,25,1225)$ satisfies the other conditions, and
$1225 = 35^2$, so there is a $35^2$-tiling of a certain $T$, which is a
``triquadratic'' tiling.  Any other tiling of that $T$ will have the same
$s=2\sin (A/4)$,  so for that tiling we would have $(M,K) = (5\lambda,
25\lambda)$ for some $\lambda\in \Z$, and $N$ would be $1225\lambda^2$, also a
square. 
\end{remark}

\begin{remark}
There will be a smallest pair $(M,K)$ such that there is a tiling with $N_1 =
2K^2-M^2$ tiles; all other tilings will have $\lambda^2 N_1$ tiles, for some
integer $\lambda$.  To see this, we consider any $(M,K)$ and $N=2K^2-M^2$ 
corresponding to a tiling.   
In general, $M/K$ will not be in lowest terms.   
 Write $s = 2\sin (A/4) = m/k$ in lowest terms.  Any tiling
of $T$ of this kind has $M/K = s$, say $(M,K) = (rm,rk)$.
The condition $K \mid M^2$ is sufficient for a tiling, but not necessary.
If it is not satisfied, we can multiply $M$ and $K$ each by $K$; then the 
sufficient condition will be satisfied and we will obtain a tiling with $K^2 N$ 
tiles, which of course will be a tiling of a larger triangle, but of the same shape. 

%
\end{remark}

\begin{proof}[Proof of Proposition~\ref{prop:group-1-last-existence}]

Assume that $2\sin(A/4)$ is rational, equal to $M/K$. Choose $R = (\alpha,
\beta, \gamma)$ with $\alpha = A/2$ and $\beta = B$.  Then the angles of $T$
are $(2\alpha,\beta, \alpha+\beta)$. By Theorem~2.4 of
\cite{laczkovich1995tilings}, there is a tiling of $T$ into $R$.

By Theorem~10 of \cite{beeson-triangletiling3}, if the tiling has $N$ tiles,
then we can write
$$ m^2 + N = 2k^2,$$
where $s = 2\sin(\alpha/2) = 2\sin(A/4)$ equals the rational number $m/k$. 
Then $N = 2k^2-m^2$.  The fraction $m/k$ is not necessarily in lowest terms,
but scaling $M$ and $K$ does not affect whether $2K^2-M^2$ is a square or not.
Hence, $N = 2k^2 - m^2$ is a square if and only
if $2K^2 - M^2$ is. 
\end{proof}

\begin{prop}
\label{prop:group-2-last-existence} Let $T$ have incommensurable angles
$(A, B, C)$, where 
$$C = 2A + B/2$$ 
with $\sqrt{3} \tan(A/2) \in \qq$.
Then,
\begin{enumerate}[{\rm (1)}]
\item $T$ has a tiling into $(\alpha, \beta, \gamma)$ with $\alpha = A$,
$\beta = B/2$, and $\gamma = 2\pi/3$.
\item For any such tiling, the number of tiles is not a square.
\end{enumerate}
\end{prop}
\begin{proof}  Let $\alpha = A$,
$\beta = B/2$, and $\gamma = 2\pi/3$.  Since $A+B+C = \pi$ and $C= 2A+B/2$,
we have $\alpha+\beta  = \pi/3$. 
By assumption and Proposition~\ref{prop:group-2-alg}, $\sqrt 3
\sin \al$ and $\cos \al$ are rational. By Theorem~2.5 of
\cite{laczkovich1995tilings}, there is a tiling of $T$ into $R$. We have
to show that the number of tiles is not a square. By Proposition
\ref{prop:group-2-alg}, we have
$$t = \frac{\tan(\alpha/2)}{\sqrt 3} \in \Q, \text{where} \ 0<t<1/3.$$
We may assume that the sides of $R$ and of $T$ are integers. Then the area of
$R$ is $\onehalf \sin \al \cd \sin \be / \sin \ga$ times a square.
The area of $T$ is $\onehalf \sin \al \cd \sin (2\be ) / \sin ((\pi /3)+\al )$
times a square. Thus the number of tiles is
\begin{align*}
N =& \frac{\sin \ga \cd \sin (2\be )}{\sin \be \cd \sin ((\pi /3)+\al )} =
\frac{\sqrt 3 \cd \cos \be}{\sin ((\pi /3)+\al )}=\\
& \sqrt 3 \cd  \frac{\cos ((\pi /3)-\al )}{\sin ((\pi /3)+\al )}=
\sqrt 3 \cd  \frac{\ha  \cos \al +\frac{\sqrt 3}{2} \sin \al}{\frac{\sqrt 3}{2} \cos \al +\ha \sin \al} =\\
& \frac{1+6t-3t^2}{1+2t-3t^2}=\frac{3t^2 -6t-1}{(t-1)(3t+1)}
\end{align*}
times a square of a rational number. By Proposition~\ref{prop:nonsquare4}, this
cannot be a square, so we indeed have a non-square tiling.
\end{proof}

\section{Usually only one tile is possible, and at most two}
\label{sec:uniqueness}

Some triangles $T$ match more than one row of the table in
Proposition~\ref{prop:TtoR}. For example, if $T$ matches the first row,
then it matches the second row as well. Indeed, if $\sqrt{3}\cd \tan (A/4)$ is
rational, then so is $\sqrt{3}\cd \tan (A/2)$, since $\tan ^2 (A/4) \in \qq$,
and thus
$$\sqrt{3}\cd \tan (A/2)=\frac{2\sqrt{3}\cd \tan (A/4)}{1-\tan ^2 (A/4)}
\in \qq .$$
In Proposition~\ref{prop:60-angle-existence} we proved that if $T$ matches the
second row with $A < B$, then it can be tiled by a tile with angles $\al, \be, 2\pi /3$,
where $\alpha = A$ and $\beta = \pi/3-A$. 
Now we show that if $T$ matches the
first row, then it has a non-square tiling using a different tile.

\begin{prop} \label{prop:60-angle-existence-2}
Let $T$ have angles $(A, B, \pi/3)$ such that 
$$\sqrt{3} \cd \tan\left(A/4\right)\in \mathbb{Q}.$$
Then
\begin{enumerate}[{\rm (1)}]
\item $T$ has a tiling into $(\alpha, \beta, \gamma)$ with $\alpha = A/2$,
$\beta = B/2$, and $\gamma = 2\pi/3$.
\item For any such tiling, the number of tiles is not a square.
\end{enumerate}
\end{prop}

\begin{proof}
Let $\alpha ,\be ,\ga$ be as in (1). Then $T=(2\alpha, 2\beta, \alpha+\beta)$.
By assumption and Proposition~\ref{prop:group-2-alg}, $\sqrt{3} \cd  \sin \al$
and $\cos \al$ are rational. By Theorem~2.5 of \cite{laczkovich1995tilings},
there is a tiling of $T$ into $R$. We will show that the number of tiles is
not a square.

By Proposition \ref{prop:group-2-alg}, the tile has commensurable sides, and
then so has $T$. We may assume that the sides of the tiles and of $T$ are
integers. Then the area of $T$ equals $\onehalf \cd \sin (2\al )\cd \sin (2\be)
/\sin \pi /3$ times a square, and the area of the tiles equals
$\onehalf \cd \sin \al \cd \sin \be /\sin 2\pi /3$ times a square. Thus the
number of tiles is
\begin{align*}
N&=\frac{\sin (2\al )\cd \sin (2\be)\cd \sin 2\pi /3}{\sin \al \cd \sin \be
  \cd \sin \pi /3}=4\cos \al \cd \cos \be \\
&=4\cos \al \cd \left( \ha \cos \al +\frac{\sqrt 3}{2} \sin \al \right)
\end{align*}
times the square of a rational. We show that $N$ is not the square of a
rational.

Put $t=\tfrac{1}{\sqrt 3} \cd \tan (\al /2)$. By Proposition
\ref{prop:group-2-alg} we find that $t\in \qq$, $t>0$, and \eqref{eq:etrig}
holds. Therefore, we have
$$N=2\cd \frac{(1-3t^2 )(1+6t-3t^2 )}{(1+3t^2 )^2}.$$
If $N$ is the square of a rational, then so is
$$M=2\cd (1-3t^2 )(1+6t-3t^2 ).$$
Let $t=a/b$, where $a,b$ are coprime positive integers. If $M$ is 
the square of a rational, then
$$P=2\cd (b^2 -3a^2 )(b^2 +6ab -3a^2 )$$
is a square. Now $P\equiv 2b^4$ (mod $3$). If $b$ is not divisible by $3$,
then $b^2 \equiv 1$ and $P\equiv 2$ (mod $3$), which is impossible. Thus
$3\mid b$. If $b=3c$, then $P=9Q$, where
$$Q=2\cd (3c^2 -a^2 )(3c^2 +6ac -a^2 ),$$
and $Q$ is also a square. Since $Q\equiv 2a^4$ (mod $3$), the argument
above gives $3\mid a$. This, however, contradicts $\gcd (a,b)=1$, and thus
$N$ cannot be the square of a rational.
\end{proof}

\begin{exmp}\label{ex1}
Let $R$ be the triangle with sides $3,5,7$, and let $\al ,\be ,\ga$ be the
angles of $R$ opposite to the sides $3,5,7$, respectively. Since $7^2 =3^2 +5^2
+3\cd 5$, we have $\cos \ga =-1/2$ and $\ga =2\pi /3$. Then
$$\sin \al =\frac{3}{7} \cd \sin \ga =\frac{3\sqrt 3}{14},$$
$\cos \al =13/14$ and $\sin (2\al )=78\sqrt{3}/14^2$.

Let $T$ be a triangle with angles $A=2\al$, $B=2\be$, $C=\pi /3$.
It is easy to check that
$$\sin (2\alpha): \sin (2\be ): \sin \tfrac{\pi}{3} =39:55:49,$$
and thus the proportion of the sides of $T$ is $39:55:49$. 

Let $S$ be a triangle with angles $(2\alpha, \tfrac{\pi}{3} -2\alpha,
\tfrac{2\pi}{3})$. The
proportion of the sides of $S$ is given by
$$\sin (2\alpha): \sin (\tfrac{\pi}{3}-2\al ): \sin \tfrac{2\pi}{3} =39:16:49.$$
Then the conditions of Propositions \ref{prop:60-angle-existence} and
\ref{prop:60-angle-existence-2} are satisfied, and thus $T$ can be tiled
with congruent triangles similar to $R$, and also
with congruent triangles similar to $S$. 

The tilings of $T$ with congruent triangles similar to $R$ and to $S$
are illustrated in Figure~\ref{figure:tiling3575} and
Figure~\ref{figure:tiling126720}. The numbers of tiles required are 3575 and 126720,
respectively.

The 126720 tiles are too many to be seen on a printed page, so the figure shows
only the triangles that should be quadratically tiled, and  the rectangles
that should be tiled by tiles oriented in the same direction. Nevertheless,
one can see that the  tiled triangles have the same shape, while the tiles
are  different.  
\end{exmp}

Next we show that there are triangles that match the third and the fourth rows
of the table in Proposition~\ref{prop:TtoR} simultaneously.

Let $\sin (A/2)=(t^2 -3)/(t^2 +3)$, where $\sqrt 3 <t<3$. Then $\sin (A/2)\in
(0,1/2)$, $A/2<\pi /6$ and $A<\pi /3$. Let the angles of the triangle $T$ be
$A$, $B=2A$ and $C=\pi -3A$. We show that if $t$ is rational, then $T$
matches both the third and the fourth rows of the table. It is obvious that $T$
matches the fourth row, as $\sin (A/2)\in \qq$. We have
$$\cos ^2 (A/2)=1- \left( \frac{t^2 -3}{t^2 +3} \right) ^2
  =\frac{12t^2}{(t^2 +3)^2} .$$
This implies $\cos(A/2)=2\sqrt{3} t/(t^2 +3)$, $\sqrt{3} \cd \cos (A/2)\in \qq$
and $\sqrt{3} \cd \tan (A/2)\in \qq$, proving that $T$ matches the third row
as well.

\begin{exmp}\label{example:ex2}
If $t=2$, then $\sin (A/2)=1/7$, $\cos (A/2)=4\sqrt{3} /7$,
$\sin A =8\sqrt{3}/49$, $\cos A=47/49$, $\sin (2A)=752\sqrt{3}/49^2$, and
$\sin (3A)=51480\sqrt{3}/49^3$. Thus
$$\sin A:\sin B:\sin C=\sin A:\sin (2A):\sin (3A)=2401:4606:6435$$
gives the proportion of the sides of $T$.

Let the triangle $R$ have angles $(A,\tfrac{\pi}{3}-A, \tfrac{2\pi}{3})$, and
let the triangle $S$ have angles $(A,\tfrac{\pi}{2}-\tfrac{3A}{2} ,
\tfrac{\pi}{2} + \tfrac{A}{2})$.
It is easy to check that the proportion of the sides of $R$ is given by
$$\sin A:\sin (\tfrac{\pi}{3}-A): \sin (\tfrac{2\pi}{3})=16:39:49,$$
and the proportion of the sides of $S$ is given by
$$\sin A:\sin (\tfrac{\pi}{2}-\tfrac{3A}{2}): \sin (\tfrac{\pi}{2} +
\tfrac{A}{2})=14:45:49.$$

By Propositions \ref{prop:A-2A-existence1} and \ref{prop:A-2A-existence2},
$T$ can be tiled by congruent triangles similar to $R$, and also
with congruent triangles similar to $S$.

These tilings provided by the proofs of Propositions
\ref{prop:A-2A-existence1} and \ref{prop:A-2A-existence2} 
use too many tiles for the individual tiles to be drawn. Therefore we have
prepared figures showing only the outline of the quadratic tilings and 
parallelograms used in the tilings.  The figures at least make clear that the same
triangle $T$ can  be tiled with a Group~1 tiling and a Group~2 tiling.
See Figure~\ref{figure:example2}. The number of tiles for the first 
tiling is $N = 31203926019578312$.%
\footnote{
This was computed directly by adding the individual quadratic tiling numbers.  We checked it with \cite[Theorem~9]{beeson-triangletiling3}.
}
For the second tiling, $N$ is more than $10^{20}$. 
Figure~\ref{figure:Prop18Details} blows up a part of that tiling to make some regions
visible that cannot be seen at the scale of Figure~\ref{figure:example2}. 
\end{exmp}

Next we show that, aside from the two exceptional cases discussed above, there
is only one possible tile that can be used to tile a given triangle $T$.

\begin{thm} \label{thm:possibleR}
Let $T$ be a non-isosceles triangle that admits a non-square, non-reptile
tiling. Then
\begin{enumerate}[{\rm (1)}]
\item If after a permutation of the angles $(A,B,C)$, we have 
$$C = \pi/3, \  A<B, \ \text{and} \ \sqrt{3} \cd \tan(A/4) \in \Q,$$
then $T$ has a non-square, non-reptile tiling using two different tiles
$(\alpha,\beta, 2\pi/3)$, with $\alpha = A$ or $\alpha = A/2$, and those are
the only such tiles, up to similarity.
\item If $B=2A$ and both $\sin(A/2)$ and $\sqrt{3} \cd \tan(A/2)$ are rational,
then $T$ has a non-square, non-reptile tiling using two different tiles $(\alpha,
\beta, 2\pi/3)$ and $(\alpha, \tfrac{3}{2} \beta, \gamma)$, and those are the
only such tiles, up to similarity. 
\item Otherwise, $T$ matches exactly one row of the table in 
Proposition~\ref{prop:TtoR}, and the angles of the tile are uniquely
determined by $T$.
\end{enumerate}
\end{thm}

\begin{proof} (1) After a permutation of $(A,B,C)$, we have $\sqrt{3} \cd
\tan(A/2)\in \Q$ and $C = \pi/3$. Since that condition holds for both angles other than 
$C$, we may choose the permutation with $A<B$.
Then $T$ matches the first two rows of the table in Proposition~\ref{prop:TtoR}, 
since $\sqrt{3}\tan(A/4) \in \Q$ implies $\sqrt{3}\tan(A/2) \in \Q$.
Then it 
can be tiled by the tiles mentioned in this proposition, as has already been
seen in Proposition \ref{prop:60-angle-existence-2}. It remains to show no
other tiling is possible; that is, $T$ does not match any other row of the
table after the second row. Remember that the permutation of $(A,B,C)$ to
match another  row might be different than the one that makes it match
the first row.

Note that, by Theorem~\ref{thm:l1995-5.3}, the angles of $T$ are incommensurable.
We will show that  in any case there will be a rational linear relation $pA+qB+rC = 0$
with $p\neq q$.  
To prove this, observe that the entries in the first column of the table, 
apart from the first two rows, can be written in the form $pA + qB + rC = 0$, 
with no two coefficients equal.
Then after a permutation of $(A,B,C)$, they still have
this form.  Then we have in the original permutation,
$$ 
\matrix 0 0 1  1 1 1 p q r  \vector A B C = \vector {\pi/3} {\pi} 0.
$$
and because $p \neq q$, the determinant is not zero. 
Therefore the equation can be solved, making the angles
all rational multiples of $\pi$. But that contradicts
 the incommensurability hypothesis.

(2) In part (2) of the theorem,  one angle of $T$ is double another.  First we show that 
$T$ does not match the second row under any permutation.  If it 
did, one angle would be $\pi/3$.  
The $\pi/3$ angle cannot be 
double or half another angle, for then the angles would be commensurable. 
Then the angles would be $(\pi/3, X, 2X)$, so $X = 2\pi/9$ and the angles
would be commensurable.  Hence, $T$ does not match the second row under
any permutation, as claimed.

Since one angle is double another, $T$ matches the third and the fourth rows. 
Then $T$ does not also match row 5 or row 6, since
(reading the third column of the table) that would
require a nontrivial linear relation between $\alpha$ and
$\beta$, contradicting incommensurability.  We are not done yet, because 
we have to consider the possibility that
$T$ matches the third and fourth rows, but with a different matching of the
angles of $T$ to $(A,B,C)$.  Since one of the angles is double another,
 there is a permutation $(p,q,r)$ of $(-2,1,0)$
such that  $pA + qB + rC = 0$. Then $p+q+r=-1$ and 
$$
\matrix 1 1 1 {-2} 1 0 p q r \vector A B C = \vector \pi 0 0.
$$
The determinant is $3r-2q - p$. By the incommensurability of the angles of $T$, this
determinant
must be zero.  But for every non-identity permutation
$(p,q,r)$ of $(-2,1,0)$,
we have $3r-2q -p \neq 0$,
contradiction.   Therefore $(p,q,r)$ is the identity
permutation, i.e., $(p,q,r) = (-2,1,0)$.  
Now both the third and the fourth lines of the table
require $A=\alpha$ and $B = 2\alpha$,  so $\alpha$ is 
the same in both tilings. Let $(\alpha,\beta,2\pi/3)$
be the tile from the third line, and $(\alpha, \beta^\prime,\gamma)$
the tile from the fourth line.  Then $C= 3\beta = 2\beta^\prime$.
Then $\pi/3 = \alpha + \beta  = \alpha + \frac 2 3 \beta^\prime$.
Then (2) of the theorem holds.

(3) $T$ matches the last two rows 
of the table (but not any of the first four).
Then,  reading the last column of the table,
we would have both $3\alpha + 2\beta = \pi$,
and $3\alpha+ 3\beta  = \pi$, and subtracting 
we have $\beta =0 $, contradiction.  
\end{proof}

\begin{proof}[Proof of Theorem \ref{thm:non-reptiling}]
Suppose a non-isosceles $T$ is tiled, and the tiling is not a reptiling. Then
by Lemma \ref{lem:T-pov-cases} and Proposition~\ref{prop:TtoR}, $T$
falls under one of the lines listed in the table of the proposition.
In Section \ref{sec:existence} we saw that in all cases, except when $T$
falls under the first or the fifth line, the tiling must be a non-square
tiling. If $T$ falls under the first line, then the same is true by
Proposition \ref{prop:60-angle-existence-2}.

Finally, if $T$ falls under the fifth line then, by Proposition
\ref{prop:group-1-last-existence}, we can have a square tiling only
if (2) of Theorem \ref{thm:non-reptiling} holds.
\end{proof}

\section{Examples}
\label{sec:examples}

The second half of our main theorem required proving the existence of certain
tilings. Our work does not require the actual construction of these tilings,
but only to prove their existence. Nevertheless, in this section, we will
actually  {\em exhibit} those tilings. These pictures are not required for
existence, but we include them for illustration.

\begin{remark}
Erd\H{o}s Problem~633 only requires us to ``classify'' the tiled triangle
$T$; but more  generally one may wish to classify the tilings themselves, with
the most general goal of understanding all triples $(T, R, N)$. In fact, a
related (also \$25, but in our opinion much harder) Erd\H{o}s Problem
\cite[p.~48]{Soifer2009-book} (indexed as Problem 634 in \cite{erdos-634}) can be seen as a
specialization of this goal, but focusing on $N$, the number of tiles,
instead of $T$ as Problem 633 does. In this light, we can see many of the
works in the literature as studying different aspects of the general problem,
such as \cite{laczkovich1995tilings} laying the groundwork for understanding
all possible $T$ and $R$ without considering $N$, or \cite{beeson-equilateral},
\cite{beeson-isosceles}, \cite{beeson-triangletiling3},
\cite{laczkovich2020rational}, and \cite{zhang2026}
studying what $(T,R,N)$ can occur if a particular choice of $T$ or $R$ is fixed.
To that end, it is interesting to view the possible tilings.  
\end{remark}

Theorem~2.5 of \cite{laczkovich1995tilings} established existence for all
these tilings. This work focused on the construction of dissections of 
a triangle $T$ into similar (rather than congruent) triangles. It was noted
that if the  sides are commensurable, then one can, by choosing a small
enough tile, refine a dissection into similar triangles to a tiling. The
proof proceeds by exhibiting sketches of dissections into similar triangles
or parallelograms, which can then be tiled by congruent triangles if the
tiles are  chosen small enough, provided the sides are commensurable. 
That method was also used by  the first and third authors in their contributions
to the existence part of our proof\footnote{
It is well-known that diagrams can be misleading 
in geometry, and that phenomenon does occur in the diagram 
of a tiling of an isosceles triangle at the top left of p.~84
in \cite{laczkovich1995tilings}, which cannot actually be drawn if
the line extending southeast from the ``center'' hits the 
base of the triangle rather than the east side as shown. One can prove
that this happens if $\al <\al _0$, where $\al _0 \approx 38.29^\circ$.
(The figure in \cite{laczkovich1995tilings} is drawn with 
$\alpha = 40^\circ$.)
This error does not affect the validity of
\cite[Theorem 2.5]{laczkovich1995tilings}; in the case $\al <\al _0$
a different construction works.
See Figure~\ref{figure:isoscelescorrect} for a correct dissection
of that triangle in this case.   Observe that the boundary between
green and red tiles in that figure does hit the base, rather than the east side
of the triangle.
}.

The examples in this section (except for the triquadratic tilings) were
constructed by the method of refining a dissection into similar triangles to
a tiling.  We started with dissections given in \cite{laczkovich1995tilings}.
But a direct application  of this method often results in tilings too large
to draw, because we have to choose a very small unit to make all the 
edges in the tiling multiples of that unit. If $N$ is more than about
fifteen thousand, the tiling cannot actually be drawn attractively on a
normal page. Luckily, we could use the method of Herdt%
\footnote{Private communication, 10 January, 2024; Herdt's method is 
explained in section 15 of \cite{beeson-isosceles} and developed more
theoretically in \cite{zhang2026}.}
to reduce the number of tiles required, so that we could produce illustrations
of all the cases in Theorem~\ref{thm:main}.

First we illustrate that every triangle has a ``quadratic tiling'' into $N^2$
tiles, for any $N$. Such a tiling is constructed by drawing $N-1$ equally
spaced lines parallel to each side of the triangle.
See Figure~\ref{figure:quadratic}, where $N=10$ produces a 100-tiling.

\begin{figure}[ht]
\centering
\includegraphics[width=0.3\linewidth]
{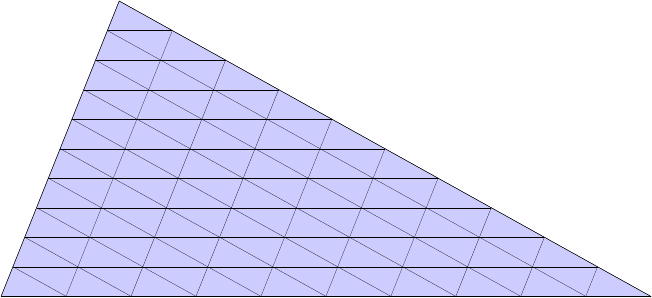}
\caption{Every triangle has quadratic tilings.}
\label{figure:quadratic}
\end{figure}

Case (1) in Theorem~\ref{thm:main} is  when $T$ is isosceles.
There are many interesting ways to tile various isosceles triangles; see
\cite{beeson-isosceles} for a full discussion of the non-equilateral case,
and \cite{beeson-equilateral} and \cite{zhang2026} 
for the equilateral case. In fact, 
\cite[Theorem 12.5, p.~134]{Soifer2009-book} shows that for every positive integer
$k$ there is an isosceles triangle $T$ that can be dissected into $km^2$
congruent triangles with a suitable $m$. Just dividing an isosceles triangle
in half is enough for our main theorem in this paper, since $2$ is not a square.
However, tilings of an isosceles triangle are used many times as components of
tilings of other triangles, so for that reason we also need to consider
tilings of isosceles triangles with incommensurable angles,
such as the one shown in Figure~\ref{figure:isoscelescorrect}.

Figure~\ref{figure:tiling27} is not necessary for our main theorem; we include it
anyway for aesthetic and historical appreciation. The second tiling in that
figure was discovered by Major Percy Alexander MacMahon (1854--1929)
\cite{macmahon} in 1921. It is one of a family of $3k^2$ tilings (the case
$k=3$). The next case is a 48-tiling, made from six hexagons (each containing
6 tiles) bordered by 4 tiles on each of 3 sides. In general one can arrange
$1+2+\ldots+(k-1)$ hexagons in bowling-pin fashion, and add $k$ tiles 
on each of three sides, for a total number of tiles of $6(1+2+\ldots+(k-1))+3k = 3(k-1)k + 3k = 3k^2$. Figure~\ref{figure:hexagonaltilings} shows more
members of this family, which we call the ``hexagonal tilings.''%
\footnote{In January, 2012, the first author bought a puzzle at the exhibition at the AMS meeting, which contained the 
second tiling in Figure~\ref{figure:tiling27} as part of a tiling of a larger hexagon. This was a surprise, as until then,
the first author thought he had discovered that tiling.}
\begin{figure}[ht]
\centering
\includegraphics[width=0.4\linewidth]{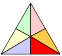} \quad
\raisebox{-2mm}{\includegraphics[width=0.4\linewidth]{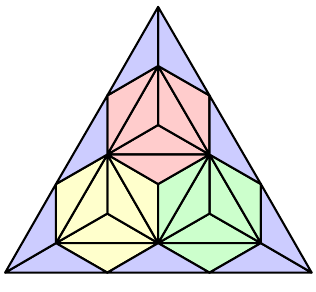}}
\caption{Some tilings with commensurable angles.}
\label{figure:tiling27}
\end{figure}

\begin{figure}[ht]
\centering
\includegraphics[width=0.4\linewidth]{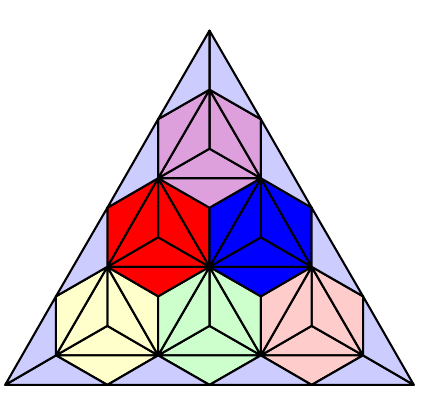}\quad
\includegraphics[width=0.4\linewidth]{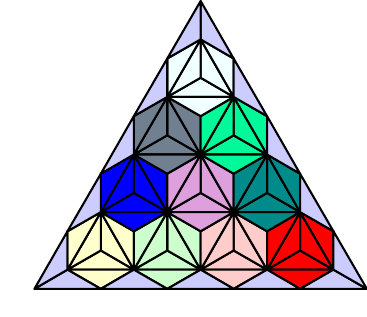}
\caption{$3k^2$ tilings with commensurable angles.} 
\label{figure:hexagonaltilings}
\end{figure}

Figure~\ref{figure:tiling1215} shows a 1215-tiling
of an equilateral triangle discovered in 2024 by
Bryce Herdt. Following the method
of \cite{laczkovich1995tilings} led to 
a tiling with 10935 tiles, just barely small enough
to draw, and reproduced in \cite{beeson-equilateral}.
Herdt reduced the number of tiles by
noticing a more efficient way to tile a 
parallelogram by cutting it into two parallelograms
tiled by triangles in different orientations. 
A construction in \cite{zhang2026} produces a family of tilings for each possible tile $(a,b,c)$, which rediscovers Herdt's tiling 
in the $(3,5,7)$ case.
We note that the tile, $(3,5,7)$,
has incommensurable angles.  

\begin{figure}
\centering
\includegraphics[width=0.5\linewidth]{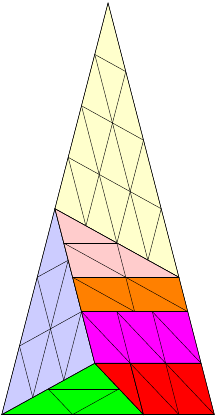}          
\caption{Vertex angle $\alpha$, base angles $\alpha+\beta$, $N=48$. The tile is $(2,3,4)$.}
\label{figure:isoscelescorrect}
\end{figure}

\begin{figure}[ht]
\centering
\includegraphics[width=0.85\linewidth]{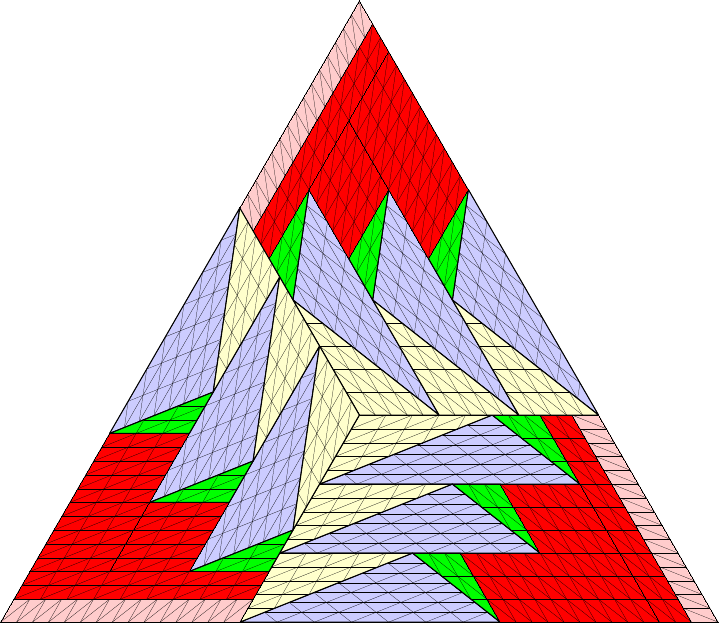}
\caption{A 1215-tiling of an equilateral triangle by $(3,5,7)$.} 
\label{figure:tiling1215}
\end{figure}

Tilings in case (2) of Theorem~\ref{thm:main}, in which the legs of a right triangle $T$ have
rational ratio, are constructed as follows. Let $N$ be a sum of two squares. Then
an appropriately-chosen right triangle can be $N$-tiled by right triangles
similar to $T$.  See Figure~\ref{figure:biquadratic}.
These tilings were first described by S.W. Golomb in \cite{golomb1964replicating}.

Since any prime congruent to $1 \pmod{4}$ is a sum of two squares, there are
infinitely many tilings in this family. These are called
{\em biquadratic tilings}.
\begin{figure}[ht]
\includegraphics[width=0.4\linewidth]{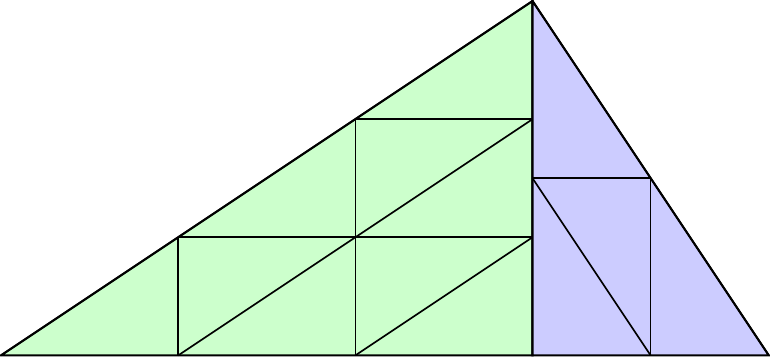}\quad
\includegraphics[width=0.4\linewidth]{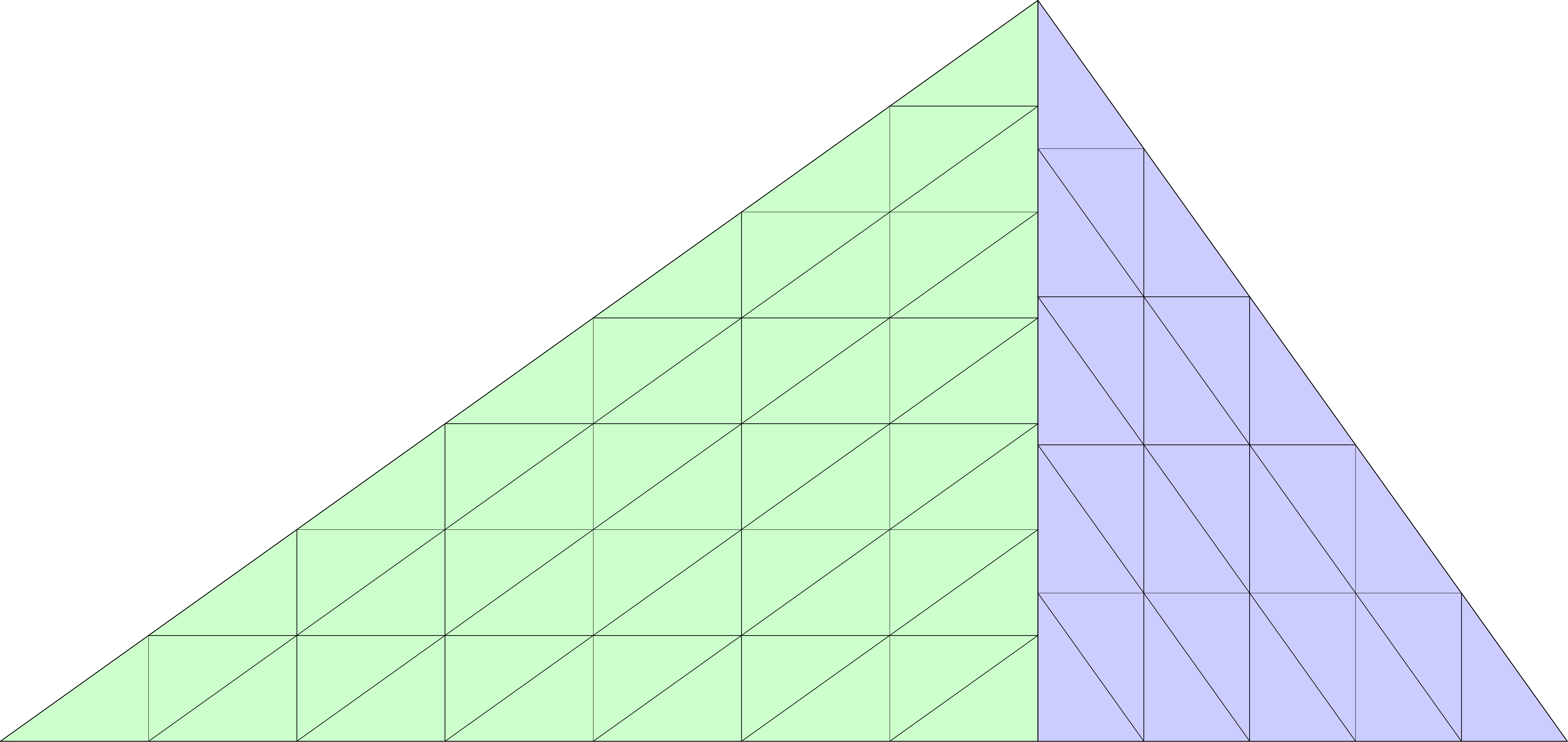}
\caption{Biquadratic tilings with $N = 13 = 3^2 + 2^2$ and $N=74 = 5^2 + 7^2$.}
\label{figure:biquadratic}
\end{figure}

Case (3) concerns tilings of the $(\pi/2, \pi/3, \pi/6)$ triangle.  See
Figure~\ref{figure:tiling3} for a tiling with $N=3$, which can be extended to tilings with $N=3k^2$ by subdivision as in reptiling.
\begin{figure}[ht]
\includegraphics[width=0.4\linewidth]{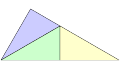}
\caption{A tiling of the 30-60-90 triangle.}
\label{figure:tiling3}
\end{figure}

That completes our samples of tilings of
triangles $T$ with commensurable angles.  

The next family, case~(4) of Theorem~\ref{thm:main}, 
contains tilings where $T$ has incommensurable angles
$(A, B, \pi/3)$, with $\sqrt{3} \tan(A/2) \in \Q$.
In the proof of Proposition~\ref{prop:60-angle-only-if}, it was shown that
if such a $T$ is tiled into $(\alpha,\beta,\gamma)$, then $\gamma = 2\pi/3$ and the angles of $T$ take one of the forms 
$$(\alpha, \beta + \pi/3, \pi/3) \text{ or }
(2\alpha, 2\beta, \pi/3).$$

We now  give examples of such tilings.  First we take up the
case $T = (\alpha, \beta + \pi/3, \pi/3)$.   To create such a tiling,
we start with a tiling of an equilateral triangle by $(\alpha,\beta,2\pi/3)$.
Then we paste a suitable quadratic tiling onto one side.  For example,
we start with  Herdt's 1215-tiling of an equilateral triangle by $(3,5,7)$.
One side of that equilateral triangle is 135, which luckily is a multiple of both 3 and 5, 
so we can tack on a quadratic 
tiling.  Depending which way we orient it, we get
a 3240-tiling or a 1944-tiling. The latter is shown
in Figure~\ref{figure:tiling1944}.  Note $\beta < \alpha$ in 
that tiling; in the 3240-tiling we have $\alpha < \beta$.
 \begin{figure}[ht]
\centering
\includegraphics[width=\linewidth]{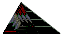}
\caption{A 1944-tiling by $(5,3,7)$. Angles of $T$  are $(\alpha,\beta +  \pi/ 3,  \pi/ 3)$.}
\label{figure:tiling1944}
\end{figure}

Next we take up the case $T = (2\alpha, 2\beta,\pi/3)$.  We 
constructed two examples.
\begin{itemize}
    \item a 7007-tiling following
    Figure~7 of \cite{laczkovich1995tilings}, and shown in Figure~\ref{figure:tiling7007};
    \item  a 3575-tiling following
    Figure~7 of \cite{zhang2026}, and shown in Figure~\ref{figure:tiling3575}.
\end{itemize}
The cited papers have different templates for tiling 
a triangle of this shape into similar triangles, and 
parallelograms. We show both examples because we found it 
surprising that  (a) there are different ways to proceed,
and (b) it makes such a difference to the final $N$.
To construct tilings (into {\em congruent} tiles),
one also has to use the method of Herdt described above, 
that is, dividing a parallelogram
into two parallelograms tiled with different orientations.
Otherwise one would have to use 25 times as many tiles, and the tilings would be too large to draw.

As remarked in Section~\ref{sec:uniqueness}, it is sometimes possible to tile
the same triangle $T$ with two different tiles.  We have illustrated this in
Figure~\ref{figure:tiling126720}, which is a tiling of a triangle similar to the
one in Figure~\ref{figure:tiling3575}, but with a different tile.
Figure~\ref{figure:tiling126720} does not show the individual tiles, because they
cannot be drawn at this scale; instead the colored areas represent quadratic
tilings or parallelogram tilings with the tiles similarly oriented.

\begin{figure}[ht]
\centering
\includegraphics[width=\linewidth]{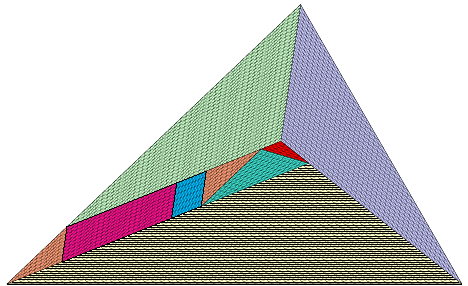}
\caption{A 7007-tiling by $(3,5,7)$. 
Angles of $T$ are $(2\alpha,2\beta,\pi/3)$.}
\label{figure:tiling7007}
\end{figure}

 \begin{figure}[ht]
\centering
\includegraphics[width=0.85\linewidth]{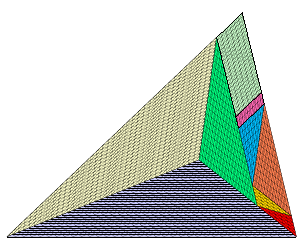}
\caption{A 3575-tiling by $(3,5,7)$. 
Angles of $T$ are $(2\alpha,2\beta,\pi/3)$.}
\label{figure:tiling3575}
\end{figure}

 \begin{figure}[ht]
\centering
\includegraphics[width=0.76\linewidth]{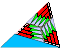}
\caption{A 126720-tiling by $(39,16,49)$. 
Angles of $T$ are the same as in Figure~\ref{figure:tiling3575}, but the tile is different.
Individual tiles are not shown here as there are too many of them.}
\label{figure:tiling126720}
\end{figure}

Next we take up case (5) of Theorem~\ref{thm:main}.
Here the triangle $T$ has angles $(\alpha,2\alpha,3\beta)$
and $\alpha+\beta = \pi/3$, so the tile
has $\gamma = 2\pi/3$.  We can make such a tiling
from any tiling of $(2\alpha, 2\beta, \pi/3)$ by tacking on another large
quadratic tiling to the right side, much as the 3240 tiling
was created from the equilateral 1215-tiling. If we start with 
the 7007-tiling, we end with a 15288-tiling, which seems
rather large. If instead we start with 
the 3575-tiling, we can draw a 6600-tiling, shown
in Figure~\ref{figure:tiling6600}.

 \begin{figure}[ht]
\centering
\includegraphics[width=0.8\linewidth]{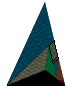}
\caption{Angles of $T$ are $(\alpha,2\alpha,3\beta)$. A 6600-tiling by $(5,3,7)$.}
\label{figure:tiling6600}
\end{figure}

Next we take up family~(6) of Theorem~\ref{thm:main}.
Here the triangle $T$ has angles $(\alpha,2\alpha,2\beta)$,
so $3\alpha+2\beta = \pi$.
An example, taken from \cite{beeson-triangletiling3}, is shown in Figure~\ref{figure:tiling77}.

Now consider the family (7) of Theorem~\ref{thm:main}.
Here $T$ has angles  $(2\alpha, \beta, \alpha + \beta)$.
 These tilings have been studied in \cite{beeson-triangletiling3}.  The history here 
is that the first author had proved there is no $7$-tiling, and was trying to prove
with the aid of a computer program that there is also no 28-tiling.  But in the 
process, he discovered that
 there {\em does} exist a 28-tiling, shown in Figure~\ref{figure:tiling28}.
  \begin{figure}[ht]
\centering
\includegraphics[width=0.5\linewidth]{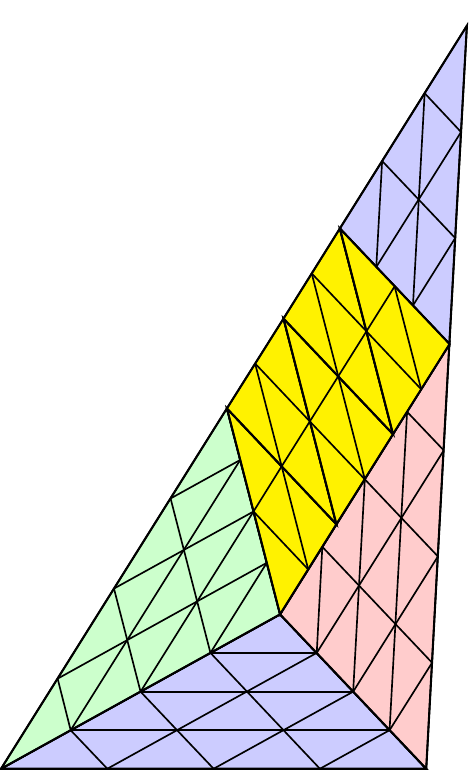}
\caption{A 77-tiling by $(2,3,4)$.}
\label{figure:tiling77}
\end{figure}
 \begin{figure}[ht]
\centering
\includegraphics[width=0.5\linewidth]{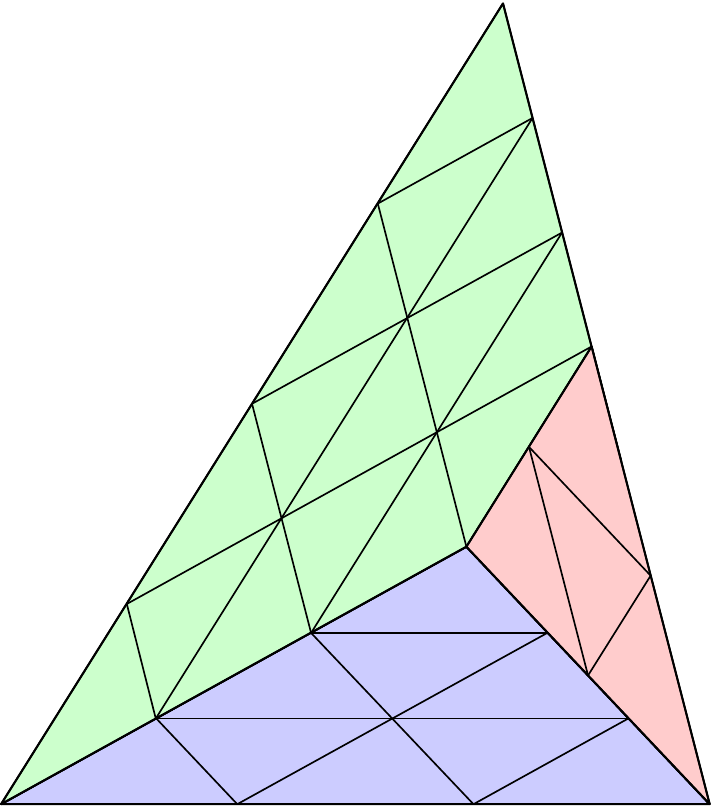}
\caption{A 28-tiling. Angles of $T$ are $(2\alpha,\beta,\alpha+\beta)$.}
\label{figure:tiling28}
\end{figure}

 The 28-tiling turned out to be the smallest member of a new 
 family of tilings.  These are the ``triquadratic tilings''; they exist when 
 the ``tiling equation'' 
 $$M^2 + N = 2K^2$$
 has a solution in integers $(K,M)$, such that 
 $K \mid N$, or equivalently $K \mid M^2$.  When we speak of a ``solution of the tiling equation'',
 we mean to include the divisibility conditions just mentioned.
 Each such solution determines the tile of the corresponding 
 triquadratic tiling: the tile must be similar to the triangle 
 with sides $a=M$, $c=K$, and $b = K-M^2/K$.  All three sides of
 that tile are integers, and the tile then satisfies the condition
 $3\alpha+2\beta = \pi$.
 
 Two other members of this family are shown in Figure~\ref{figure:tiling153}.
  \begin{figure}[ht]
\centering
\includegraphics[width=0.3\linewidth]{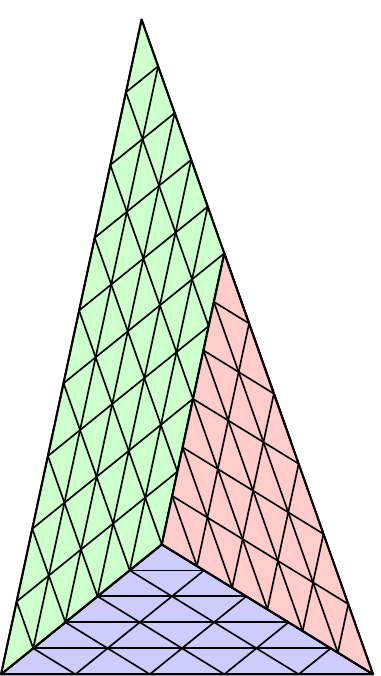}\quad
\includegraphics[width=0.6\linewidth]{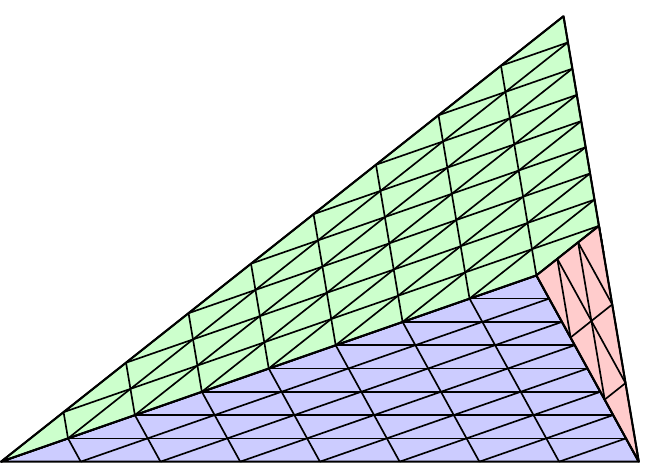}
\caption{A 126-tiling and a 153-tiling.}
\label{figure:tiling153}
\end{figure}
All the tilings in this family of triquadratic tilings belong to Group~1  and tile
a triangle $T$ with angles $(2\alpha,\beta, \alpha +\beta)$, as can be seen in the figures.
\smallskip

Now we come to the final Case~(8) of Theorem~\ref{thm:main}.  We follow the 
plan of Figure~11 in \cite{zhang2026}.%
\footnote{
The construction suggested by Figure~7 of 
\cite{laczkovich1995tilings} leads to a tiling with a quarter of a million tiles.}
Here
the triangle $T$ has angles $(\alpha,2\beta, 2\alpha+\beta)$,
and the tile has $\gamma = 2\pi/3$. 
See Figure~\ref{figure:tiling7128}. 

 \begin{figure}[ht]
\centering
\includegraphics[width=\linewidth]{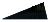}
\caption{A 7128-tiling. Angles are $(\alpha,2\beta,2\alpha + \beta)$. 
 The tile is $(3,5,7)$.}
\label{figure:tiling7128}
\end{figure}

\begin{figure}[ht]
\centering
\includegraphics[width=0.45\linewidth]{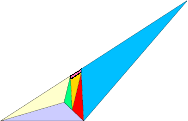}
\includegraphics[width=0.45\linewidth]{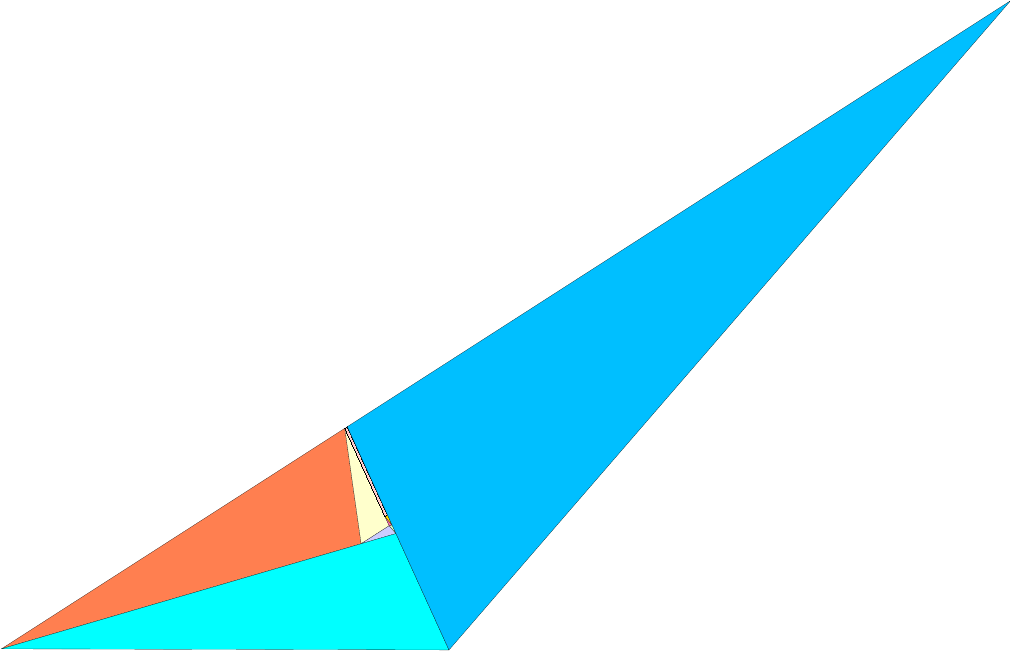}
\caption{The same triangle tiled by two different tiles. See
Example~\ref{example:ex2}.}
\label{figure:example2}
\end{figure}


\begin{figure}[ht]
\centering
\includegraphics[width=0.45\linewidth]{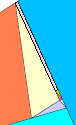}
\caption{Details from the second tiling in Figure~\ref{figure:example2}.}
\label{figure:Prop18Details}
\end{figure}

\FloatBarrier


\begingroup
\raggedright

\endgroup

\end{document}